\documentclass[hidelinks,onefignum,onetabnum]{siamart250211}
\usepackage{amssymb, amsmath}
\usepackage{tikz}
\usepackage[utf8]{inputenc}
\usepackage{graphicx} 
\usepackage[compatibility=false]{caption}  
\usepackage{subcaption} 
\usepackage{todonotes}
\newcommand{\Hm}[1]{\leavevmode{\marginpar{\tiny%
			$\hbox to 0mm{\hspace*{0.5cm}$\leftarrow$\hss}%
			\vcenter{\vrule depth 0.1mm height 0.1mm width \the\marginparwidth}%
			\hbox to 0mm{\hss$\rightarrow$\hspace*{-0.5mm}}$\\\relax\raggedright
			#1}}}
\usepackage{makecell} 
\usepackage[table]{xcolor} 
\usepackage{tcolorbox}
\usepackage{cleveref}
\usepackage{diagbox}
\usepackage{siunitx}   
\usepackage{caption}
\usepackage{multirow, booktabs}
\usetikzlibrary{arrows.meta, positioning}
\usetikzlibrary{decorations.pathmorphing, decorations.text}
\usetikzlibrary{decorations.pathreplacing, positioning}

\usepackage{lipsum}
\usepackage{amsmath}
\usepackage{amsfonts}
\usepackage{graphicx}
\usepackage{epstopdf}
\usepackage{algorithmic}
\ifpdf
  \DeclareGraphicsExtensions{.eps,.pdf,.png,.jpg}
\else
  \DeclareGraphicsExtensions{.eps}
\fi

\newsiamremark{remark}{Remark}
\newsiamremark{example}{Example}
\newsiamremark{assumption}{Assumption}
\crefname{assumption}{Assumption}{Assumption}
\newsiamthm{claim}{Claim}
\newsiamremark{fact}{Fact}
\crefname{fact}{Fact}{Facts}

\headers{Data assisted iterative regularization method}{}

\title{On the convergence of iterative regularization method assisted by the graph Laplacian with early stopping \thanks{Submitted to the editors DATE.
\funding{The work of HB and AKG was funded by the ANRF through grant CRG/2022/00549.}}}

\author{
Harshit Bajpai\thanks{Department of Mathematics, Indian Institute of Technology Roorkee, Roorkee, 247667, India 
  (\email{harshit\_b@ma.iitr.ac.in}, \email{ankik.giri@ma.iitr.ac.in}).}
\and
Gaurav Mittal\thanks{Defence Research and Development Organization, Near Metcalfe House, Delhi, 110054,
India 
  (\email{gaurav.mittaltwins@yahoo.com}).}
\and
Ankik Kumar Giri\footnotemark[2]
}

\usepackage{amsopn}

\ifpdf
\hypersetup{
  pdftitle={On the convergence of iterative regularization method assisted by graph Laplacian with early stopping},
  pdfauthor={Harshit Bajpai, Gaurav Mittal and Ankik Kumar Giri}
}
\fi

\begin{document}

\maketitle
\begin{abstract}
We present a data-assisted iterative regularization method for solving ill-posed inverse problems. The proposed approach, termed \texttt{IRMGL+\(\Psi\)}, integrates classical iterative techniques with a data-driven regularization term realized through an iteratively updated graph Laplacian. Our method commences by computing a preliminary solution using any suitable reconstruction method, which then serves as the basis for constructing the initial graph Laplacian. The solution is subsequently refined through an iterative process, where the graph Laplacian is simultaneously recalibrated at each step to effectively capture the evolving structure of the solution. A key innovation of this work lies in the formulation of this iterative scheme and the rigorous justification of the classical discrepancy principle as a reliable early stopping criterion specifically tailored to the proposed method. Under standard assumptions, we establish stability and convergence results for the scheme when the discrepancy principle is applied. Furthermore, we demonstrate the robustness and effectiveness of our method through numerical experiments utilizing four distinct initial reconstructors $\Psi$: the adjoint operator (Adj), filtered back projection (FBP), total variation (TV) denoising, and standard Tikhonov regularization (Tik). It is observed that \texttt{IRMGL+Adj} demonstrates a distinct advantage over the other initializers, producing a robust and stable approximate solution directly from a basic initial reconstruction.

\end{abstract}



\section{Introduction}
This work is concerned with the analysis and numerical solution of a linear ill-posed inverse problem, modeled by the operator equation  
\begin{equation}\label{Model eqn}
    Au = v,
\end{equation}  
where \( A: \mathcal{D}(A) \subset \mathcal{U}\simeq \mathbb{R}^n \to \mathcal{V} \simeq \mathbb{R}^m \) is a discretized version of a linear operator between \( \mathcal{U} \) and \( \mathcal{V} \), and \( \mathcal{D}(A) \) denotes the domain of \( A \). Unless otherwise specified,  we assume that both spaces are equipped with the standard inner product \( \langle \cdot, \cdot \rangle \) and corresponding norm \( \|\cdot\| \).
 Given a fixed $u^\dagger$ and the corresponding exact data $v:= Au^\dagger$ our goal is to recover a good approximation of  $u^\dagger$ from noisy observations $v^\delta$ to $v$ (unknown in practical applications) which satisfies 
\begin{equation}\label{eqn: delta}
    \|v^\delta - v\| \leq \delta
\end{equation}
with the known noise level $\delta >0.$ In practical applications, (\ref{Model eqn}) is typically ill-posed, meaning that even small perturbations in the data can lead to disproportionately large deviations in the resulting solution. Therefore, regularization techniques are essential to ensure the stability and reliability of the approximate solution.
Many of these approaches can be classified as Tikhonov-type variational methods, which are typically formulated as
\begin{equation}\label{eqn: variational}
u_\alpha^\delta := \arg\min_{u \in  \mathcal{U}} \{ \mathfrak{D}(Au ; v^\delta) + \alpha \mathfrak{R}(u) \},
\end{equation}
where \(u \in \mathcal D (A), \mathfrak{D}(\cdot\ ; \cdot)\) denotes a pseudo-distance measuring the discrepancy between the forward model \(Au\) and the observed (possibly noisy) data \(v^\delta\). The term \(\mathfrak{R}(u)\) serves as a regularization functional, and the regularization parameter \(\alpha > 0\) governs the trade-off between the influence of the data misfit and regularization. 
A common choice for these terms is
\begin{equation*}
\mathfrak{D}(Au; v^\delta) = \frac{1}{2} \|Au - v^\delta\|_2^2, \quad \mathfrak{R}(u) = \|u\|_1,
\end{equation*}
where the \(\ell_2\)-norm ensures a least-squares fit to the data, and the \(\ell_1\)-norm promotes sparsity in the solution. For further details, we refer the reader to \cite{engl1996regularization} and \cite{scherzer2009variational}.

\noindent
The choice of the regularization functional $\mathfrak{R}$ plays a pivotal role in the success of the reconstruction. When prior knowledge about the true solution is available, $\mathfrak{R}$ can be specifically designed to incorporate this information, thereby steering the regularization process toward a restricted class of solutions that exhibit the desired features or structures. Few prominent examples of regularization functionals are as follows.
\begin{itemize}
    \item \(\mathfrak{R}(u) = \|u - u^{(0)}\|_2^2\), where $\|\cdot\|_2$ is the Euclidean norm. This choice is particularly effective when the prior information \(u^{(0)}\) about the true solution is available. For further details, see \cite{bajpai2024hanke,mittal2025convergence,scherzer1998modified}.
    \item \(\mathfrak{R}(u) = \|Lu\|_q^q\), where $q \geq 1$, \(\|\cdot\|_q\) is the Euclidean \(q\)-norm and  \(L\) denotes a linear differential operator. This choice is especially valuable in imaging applications, as differential operators are capable of capturing sharp variations such as edges or intensity discontinuities. See \cite{hansen2006deblurring, jain1989fundamentals, ng1999fast} for the references.
    \item  \(\mathfrak{R}(u) = \|Mu - v^\delta\|_2^2\), where \(M\) is a data-driven operator constructed from a prior dataset \(\{(u^{(i)}, Au^{(i)})\}_{1 \leq i \leq l}\), such that \(Mu^{(i)} := Au^{(i)}\) for each \(i\). This approach is particularly useful in scenarios where the underlying physical or technical model is not fully understood, as it leverages empirical data to inform the regularization process. See \cite{arridge2019solving,aspri2020data,aspri2020data1, bajpai2025stochastic} for the references.
\end{itemize}
Beyond these methods, considerable research has focused on learning-based regularization techniques. Several approaches incorporate neural networks directly into the regularizer, for example, networks trained for denoising~\cite{Meinhardt2017,Romano2017} or artifact removal~\cite{Gonzalez2019,Li2020,Obmann2020}, thereby promoting reconstructions invariant under the learned transformations. More recently, adversarial regularization~\cite{Lunz2018} employs discriminative networks to distinguish between artifact-free and artifact-corrupted images. A comprehensive overview of deep learning--based regularization frameworks for inverse problems is provided in~\cite{arridge2019solving}.

Digital images, formed by pixels organized on a two-dimensional grid, naturally lend themselves to a graph-based representation. In light of this, graph-based differential operators \(\Delta\) have emerged in recent years as effective substitutes for traditional Euclidean differential operators \(L\), particularly in image processing, image denoising and computed tomography (CT) problems, see \cite{bianchi2021graph, bianchi2023graph,bianchi2024improved, bianchi2025data, buccini2021graph, gilboa2007nonlocal,gilboa2009nonlocal,lou2010image, peyre2011non,zhang2010bregmanized}. Graph-based operators have demonstrated consistently strong performance, largely due to their ability to more accurately capture the complex structures and textures inherent in image data. In contrast to traditional Euclidean differential operators, which rely primarily on spatial proximity, graph-based methods incorporate both spatial and intensity similarities between pixels.  This enables a more nuanced representation of image content and facilitates richer information extraction. For a detailed discussion on graph-based operators, the reader is referred to Section~\ref{sec:prem}.

\noindent
A critical aspect in graph-based regularization methods is the construction of the graph from a signal that effectively captures the key features of the ground truth $u^\dagger$. As highlighted in \cite{lou2010image}, directly constructing the graph $S$ from the observed, noisy data $v^\delta$ often leads to suboptimal performance in imaging tasks such as tomographic reconstruction and deblurring. This shortcoming arises because $v^\delta$ resides in a different domain than $u^\dagger$. 
To mitigate this issue, the authors proposed a preprocessing step that transforms $v^\delta$ into a more suitable representation $\Psi(v^\delta)$, from which the graph is then constructed. The mapping $\Psi : \mathcal{V} \rightarrow \mathcal{U}$ serves as a reconstruction operator that lifts the observations from  $\mathcal{V}$ into $\mathcal{U}$, where the ground truth signal lives. This transformation can be realized using standard techniques such as Tikhonov filtering \cite{engl1996regularization} or the filtered backprojection (FBP) method \cite{kak2001principles}, depending on the specific inverse problem.
In a similar vein, and independently, iterative schemes that adaptively update the graph weights during the reconstruction process were proposed in \cite{arias2009variational,bianchi2024improved,peyre2008non, peyre2011non, zhang2010bregmanized}.

Despite providing a coherent and principled optimization framework for solving inverse problem (\ref{Model eqn}), variational regularization methods of the form (\ref{eqn: variational}) have a significant number of drawbacks in comparison to iterative methods. The selection of an appropriate regularization parameter \(\alpha\) remains a critical and often challenging task in variational methods, with poor choices leading to suboptimal reconstructions. Furthermore, these methods tend to be computationally intensive, particularly in high-dimensional settings. In contrast, iterative regularization methods inherently offer greater flexibility through adaptive regularization via early stopping rules. They are typically easier to implement, more scalable, and naturally suited for incorporating data-driven components.

\noindent To effectively address these challenges, we propose a novel iterative scheme, referred to as the \textit{Iterative Regularization Method assisted by the Graph Laplacian} (\texttt{IRMGL+\(\Psi\)}, see also Fig. \ref{fig:abstract illustration}), which performs the following iterative updates
\begin{equation}\label{main iterative schrme}
\begin{cases}
     u_{k+1}^\delta = u_k^\delta - \alpha_k^\delta A^*(Au_k^\delta - v^\delta) - \beta_k^\delta \Delta_{u_k^\delta}u_k^\delta, \quad k \in \mathbb{N},\\
     u_0^\delta \hspace{3.8mm} = \Psi(v^\delta),     
\end{cases}
\end{equation}
where $\Psi: \mathcal{V} \to \mathcal{U}$ is the initial reconstructor,  \(A^*\) denotes the adjoint of \(A\), \(\Delta_u\) is a graph Laplacian build from \(u\) (see Subsection \ref{subsec: Images and Graphs} for a proper definition) and  \(\alpha_k^\delta > 0\), \(\beta_k^\delta \geq 0\) are suitably chosen step sizes and weighted parameters, respectively. The detailed discussion on the method can be found in Subsection \ref{subsec: The method}. Note that, this method is obtained from a gradient descent applied on a functional 
\[J(u):= \frac{1}{2} \left(\|Au - v^\delta\|^2 + \beta \langle u, \Delta_u u\rangle \right).\]
where $\beta \geq 0$ is a constant. It is crucial to highlight that the proposed  scheme  (\ref{main iterative schrme}) serves as the iterative counterpart to the existing \texttt{graphLa+$\Psi$}~\cite{bianchi2025data} and \texttt{it-graphLa \(\Psi\)}~\cite{bianchi2024improved} frameworks. 

\noindent Both \texttt{graphLa+$\Psi$} (which is based on a variational framework) and \texttt{it-graphLa\(\Psi\)} (developed for acoustic impedance inversion by incorporating an iterated graph Laplacian into a variational regularization framework) require solving a minimization problem at each step. While \cite{bianchi2024improved} provided numerical evidence for their approach, a rigorous mathematical justification was not established. In contrast, our proposed \texttt{IRMGL+\(\Psi\)} method presents a distinct advancement: it offers an inherently iterative and robust solution by circumventing the need to solve computationally intensive minimization problems, a characteristic of the aforementioned works \cite{bianchi2024improved, bianchi2025data}. Crucially, this work also addresses a critical theoretical void by furnishing comprehensive convergence and stability proofs for our scheme. The details on the implementation of these two methods are provided in the \Cref{appendix} to enable a comparative visualization with our  scheme.
Moreover, \texttt{IRMGL+\(\Psi\)} can also be interpreted as a generalization of the classical Landweber method \cite{hanke1995convergence}, where the extension arises from the incorporation of a data-driven term induced by the graph Laplacian.

\noindent Due to the ill-posedness of (\ref{Model eqn}), the iterative process must not be carried out indefinitely. Instead, it requires a suitable termination strategy, a concept widely known as \emph{early stopping} in the field of  inverse problems and machine learning~\cite{engl1996regularization}. It is now widely recognized that early stopping is a critical aspect that must be carefully addressed in the design of image reconstruction models \cite{barbano2023image, jahn2024early,wangearly}.

\noindent To serve the purpose, we employ the discrepancy principle~\cite{morozov1966solution}, which suggests that the iteration should be terminated once the residual norm becomes comparable to the noise level \( \delta \).  Hence the stopping index \(k_\delta = k(\delta, v^\delta) \) determined by the discrepancy principle is defined as
\begin{equation}\label{eqn: discrepancy}
  k_\delta := \min \left\{ k \geq 0 : \left\| Au_k^\delta- v^\delta \right\| \leq \tau \delta \right\}, 
\end{equation}
where \( \tau > 1 \) is a known constant. This principle is applicable whenever concrete value of \( \delta \) (noise level) is available, which can often be inferred directly from the data. The above equation outputs \(k_\delta\), so we use \(u_{k_\delta}^\delta\) as an approximate solution of (\ref{Model eqn}). We investigate the convergence behavior of the \texttt{IRMGL+\(\Psi\)} method under the discrepancy principle~\eqref{eqn: discrepancy} and we establish that the discrepancy principle guarantees termination of the iteration after a finite number of steps. Furthermore, we introduce a systematic criterion for choosing the parameters $\alpha_k^\delta$ and $\beta_k^\delta$. The proposed parameter selection strategy is adaptive in nature, utilizing only information that is intrinsically available during the iterative process.

\noindent It is worth noting that the data-driven term in \texttt{IRMGL+\(\Psi\)} method (\ref{main iterative schrme}) inherently depends on both the noise level~$\delta$ and the observed data~$v^\delta$. This dual dependence introduces additional complexity into the analysis, rendering it highly nonstandard from existing iterative regularization methods \cite{kaltenbacher2008iterative}.
\begin{figure}
    \centering
\begin{tikzpicture}[>=Stealth, node distance=2cm]\label{abstract illustration}

\draw[thick] (0,0) ellipse (1.5cm and 2.5cm);
\draw[thick] (6,0) ellipse (1.5cm and 2.5cm);
\draw[fill=gray!30, thick] (-0.2,-1.8) circle (0.5cm);


\node[circle, fill=black, inner sep=1pt, label=left:$\Psi(v^\delta)$] (psi delta) at (-0.1,1.5) {};
 \node[circle, fill=black, inner sep=1pt, label=left:$u_1^\delta$] (u1) at (0.2,0.9) {};
 \node[circle, fill=black, inner sep=1pt, label=left:$u_2^\delta$] (u2) at (0.2,0.1) {};
  \node[circle, fill=black, inner sep=1pt, label=right:$u_{k_\delta}^\delta$] (uk delta) at (0.2,-1.0) {};
 \node[circle, fill=black, inner sep=1pt, label=right:$\Bar{u}$] (bar u) at (-0.9,-1.2) {};   
 
\node[circle, fill=black, inner sep=1pt, label=left:$u$] (u) at (-0.23,-2) {};

\node[circle, fill=black, inner sep=1pt, label=right:$v^\delta$] (vdelta) at (5.5,1.5) {};
\node[circle, fill=black, inner sep=1pt, label=right:$v$] (v) at (5.9,0.2) {};
 \node (mid) at (-1.0,0.5) {};
\draw[->] (vdelta) to node[midway, above] { $\hspace{2mm} \Psi$ } (psi delta) ;
\draw[->, thick, bend right=40] (u) to node[midway, above] { $\hspace{2mm} A$ } (v);
\draw[->] (v) to[bend right=20] node[midway, above] { $\hspace{2mm} \delta$ } (vdelta);
\draw[dotted, bend left=40, ->] (psi delta) to (u1) ;
\draw[dotted, bend left=40, ->] (u1) to (u2) ;
\draw[densely dotted, ->] (u2) to (uk delta) ;
\draw[decorate, decoration={coil, aspect=0, amplitude=2mm, segment length=4mm}, ->] (uk delta) to node[midway, right] { \tiny $\hspace{2mm}\delta \to 0$ } (u);
\draw[decorate, decoration={zigzag, amplitude=2mm, segment length=8mm}, ->]
    (psi delta) .. controls(mid) .. (bar u);
     \node[align=center, right=1.2cm of u2] {\texttt{IRMGL+\(\Psi\)}};
    \draw[decorate, decoration={zigzag, amplitude=2mm, segment length=8mm}, ->]
(psi delta) .. controls(mid) .. node[midway, above] {\tiny $ \hspace{-0.4cm} \delta \to 0$} (bar u); 

\node at (0,-3) {$\mathcal{U}$};
\node at (6, -3) {$\mathcal{V}$};

\end{tikzpicture}
\caption{An abstract illustration of the \texttt{IRMGL+\(\Psi\)} method. The initial reconstructor $\Psi$ does not necessarily serve as a regularization method, as indicated by the piecewise linear trajectory of $\Psi(v^\delta)$ as $\delta \to 0$. However, when combined with \texttt{IRMGL+\(\Psi\)}, the resulting process constitutes a stable and convergent regularization method. The dotted trajectory represents the iterative progression of the method, while the approximate solution at iteration $k_\delta$ is denoted by $u_{k_\delta}^\delta$. The coiled path illustrates the regularization behavior of the method, highlighting the convergence $u_{k_\delta}^\delta \to u$ as $\delta \to 0$. Full mathematical justification for the convergence is given in subsequent sections. The difference between \texttt{IRMGL+\(\Psi\)} and \texttt{GraphLa+$\Psi$} can be noted using  \cite[Fig. 1.1]{bianchi2025data}.
}
    \label{fig:abstract illustration}
\end{figure}
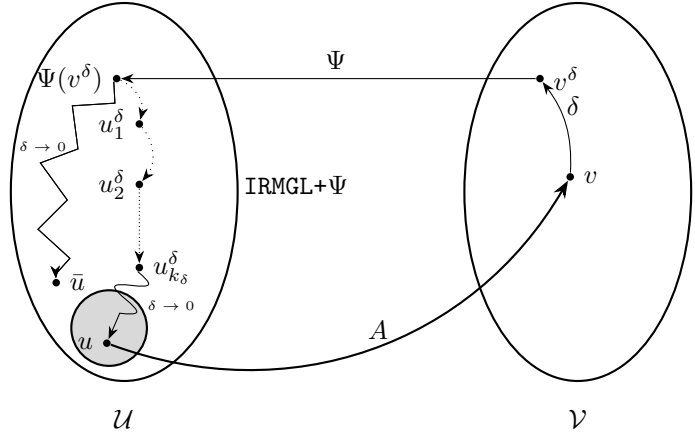

\noindent The paper is structured as follows.  Section~\ref{sec:prem} introduces the necessary notation and provides preliminaries related to graphs and image representations. In Section~\ref{sec:IRMGP}, we describe the proposed method in detail, including an algorithmic formulation, and formulate the conditions under which convergence is ensured. We also prove the stability and convergence of the method within this section. Section~\ref{sec: numerical} is dedicated to demonstrating the effectiveness of the proposed approach through a series of numerical experiments, highlighting its advantages over variational-type methods assisted by graph Laplacian regularization. Finally, we summarize our contributions and discuss potential directions for future research in Section \ref{sec:conclusions}.


\section{Preliminaries}
\label{sec:prem}
In this section, we present foundational definitions and results that are related to this work. We begin in subsection \ref{subsec: Graph Theory} by outlining essential concepts from graph theory. One may refer to  \cite{keller2021graphs} for more details.

\subsection{Graph theory}\label{subsec: Graph Theory}
\begin{definition}
For a finite set $S$, we define a graph over $S$ as a pair $G=(S, w)$, where $w: S\times S \to [0, \infty) $ is called edge-weight function if for every $a, b \in S$, it satisfies
\begin{itemize}
    \item Symmetry: $w(a, b)= w(b,a),$
    \item No self-loops: $w(a, a) =0.$
\end{itemize}
\end{definition}
The set of edges is given by $T: = \{(a, b) \in S\times S \hspace{1mm} | \hspace{1mm} w(a, b) \neq 0\}$. Two nodes $a$ and $b$ are said to be connected if $w(a, b) >0$. In such cases, we denote this relation by $a \sim b.$  This weight \( w(a, b) \) measures how strongly nodes \( a \) and \( b \) are connected.
\begin{definition}[Graph Laplacian] For any function $\textbf{x}: S 
\to \mathbb{R},$ the graph Laplacian $\Delta \textbf{x} : S \to \mathbb{R}$ associated to the graph $G = (S, w)$ is defined by the action
\begin{equation}\label{Eqn: Graph Laplacian}
\Delta \textbf{x} (a):= \sum_{a \sim b} w(a, b)(\textbf{x}(a) - \textbf{x}(b)).
\end{equation}
\end{definition}
    The corresponding graph Laplacian matrix for the operator $\Delta$ is given by 
    \begin{equation}\label{Delta = D-W}
        \Delta := D - W,
    \end{equation}
  where $W  = \begin{bmatrix}
        w(a_i, b_j)
    \end{bmatrix}_{\{i, j \in |S|\}}$ is the \emph{weight matrix} and $D$ is the \emph{diagonal degree matrix} defined as  
\[
    D = \operatorname{diag}\bigl(d(a_1),\dots,d(a_{|S|})\bigr), 
    \qquad d(a) := \sum_{\substack{b \in S \\ (a,b) \in T}} w(a,b).
\]
    Here $|S|$ denotes the cardinality of the set $S$. It can be noted that $\Delta$ defined in (\ref{Delta = D-W}) applied on $\textbf{x}(a)$ will give the same action as in (\ref{Eqn: Graph Laplacian}). 
\subsection{Images and graphs}\label{subsec: Images and Graphs}
This subsection goes over the fundamental ideas needed to create a graph from an image. Notably, defining a graph necessitate an edge-weight function $w$ and a set of nodes $S$.

\noindent The union of multiple pixels \(a \in S\) arranged on a grid forms an image. It is therefore natural to represent each pixel as an ordered pair
\[
    a = (i_a, j_a),
\]

where \(i_a = 1, \ldots, p\) and \(j_a = 1, \ldots, q\).  
In this context, \(p\) and \(q\) denote the total number of pixels along the horizontal and vertical directions, respectively.  
The index \(i_a\) specifies the horizontal position of the pixel within the grid,  
while \(j_a\) specifies the vertical position. 
Thus, the node can be configured as 
 \begin{equation*}
     S=\{a \hspace{1mm} | \hspace{1mm} a=(i_a, j_a), i_a=1, \ldots, p, \; j_a = 1, \ldots , q\}.
 \end{equation*}
 
\noindent For simplicity, we consider a grayscale image, which is characterized by the intensity values of its pixels. Such an image can be represented as a function that assigns an intensity value to each pixel location, i.e.,
\begin{equation*}
    \textbf{x} : S \to [0, 1],
\end{equation*}
where $0$ corresponds to black, $1$ corresponds to white, and intermediate values represent varying shades of gray. A widely used method for connecting two pixels by $w$ based on their spatial proximity and light intensity is provided by
\begin{equation}\label{eqn: edge weight fun}
    w_{\textbf{x}}(a, b) = \underbrace{g(a, b)}_{geometry}\times\underbrace{h_{\textbf{x}}(a, b)}_{\textbf{x}- intensity},
\end{equation}
where $g(a, b)$ denotes an edge weight function determined by the geometric characteristic of $S.$ A commonly used form of $g$ is 
\begin{equation*}\label{g(a,b)}
   g(a, b) = \mathbf{1}_{(0, R]}(\eth (a, b)) 
\end{equation*}
where \( \mathbf{1}_{(0, R]} \) denotes the indicator function,  \( R>0 \) is a control parameter specifying the maximum allowable distance for two pixels to be considered neighbors, and \( \eth(a, b) \) represents the distance between pixels \( a \) and \( b \).
 The common choices are $\eth(a, b)= |i_a - i_b| + |j_a - j_b|$ or $\eth(a, b) = \max \{|i_a - i_b|, |j_a - j_b|\}$. The function $h_{\textbf{x}}(a, b)$ serves to weight the connection between two pixels based on their intensity values and for that, in general, Gaussian kernel function is used. That is 
\begin{equation*}
    h_{\textbf{x}}(a, b) := \exp \left (\frac{-|\textbf{x}(a) - \textbf{x}(b)|^2}{\sigma} \right),
\end{equation*}
where $\sigma>0$ is the control parameter. We refer to \cite{bronstein2017geometric, calatroni2017graph,gilboa2009nonlocal, lou2010image} for a thorough description of how to define the weights of the graph Laplacian using such a Gaussian kernel function. A similar construction of a graph induced by a grayscale image is given in \cite{bianchi2024improved, bianchi2025data}.
Assume that $g(a,b) = \mathbf{1}_{(0,R]}(\eth(a,b))$ has a finite interaction radius $R>0$, then each node $a\in S$ is connected to at most $N$ neighbours, that is,
\begin{equation}\label{eqn:degree}
    \deg(a) := \#\{b \in S : g(a,b) \neq 0\} \le N \qquad \text{for all } a \in S,
\end{equation}
where $ \deg(a)$ denotes the degree of the graph node and $N$ is independent of the graph size $|S|$. 
In the next example, we illustrate the basic concepts of graph construction and the corresponding graph Laplacian using a minimal \(2 \times 2\) image.  
This example allows visualization of the graph structure, highlights the effect of parameters such as the neighborhood radius \(R\) and weight scaling \(\sigma\), and provides intuition for the role of graph-based operators in regularization.  It thus serves as a foundation for the more complex image reconstruction scenarios discussed later.
\begin{example}\label{exp: first}
    Consider a $2\times2$ grayscale image with pixel intensities given as in Table~\ref{Tab:Example}.
\begin{table}[htbp]
\centering
\begin{tabular}{|c|c|c|}
\Xhline{1.2pt}
\diagbox[width=4em]{$i$}{$j$} & 0   & 1   \\
\Xhline{0.8pt}
0                             & 0.2 & 0.3 \\
\Xhline{0.8pt}
1                             & 0.5 & 0.1 \\
\Xhline{1.2pt}
\end{tabular}
\caption{Pixel values indexed by $(i, j)$}
\label{Tab:Example}
\end{table}
This image corresponds a function $\textbf{x}: S \to [0, 1],$ where each pixel $a = (i,j) \in S$ has an intensity value $\textbf{x}(a).$ Here $S= \{a_1 = (0,0), a_2 = (0,1), a_3 = (1, 0), a_4 = (1,1)\}$ and we will use (\ref{eqn: edge weight fun})  to define the edge weight function with $\eth(a, b)= |i_a - i_b| + |j_a - j_b|$, $R=1$ (so only immediate neighbors are connected) and $\sigma = 0.01$ (for visible weight variation). Under these conditions, the weight matrix $W$ and diagonal degree matrix $D$ are given as 
\[
W =
\begin{bmatrix}
0         & 0.3679 & 0.0001 & 0         \\
0.3679 & 0         & 0         & 0.0183 \\
0.0001 & 0         & 0         & 0.0000 \\
0         & 0.0183 & 0.0000 & 0
\end{bmatrix}, D =
\begin{bmatrix}
0.3680 & 0          & 0          & 0          \\
0          & 0.3861 & 0          & 0          \\
0          & 0          & 0.0001 & 0          \\
0          & 0          & 0          & 0.0183
\end{bmatrix}.
\]
 Further, using (\ref{Delta = D-W}), we can calculate the graph Laplacian $\Delta_{\textbf{x}}$ of $\textbf{x}.$

   \begin{figure}
  \centering
  \includegraphics[width=0.38\textwidth]{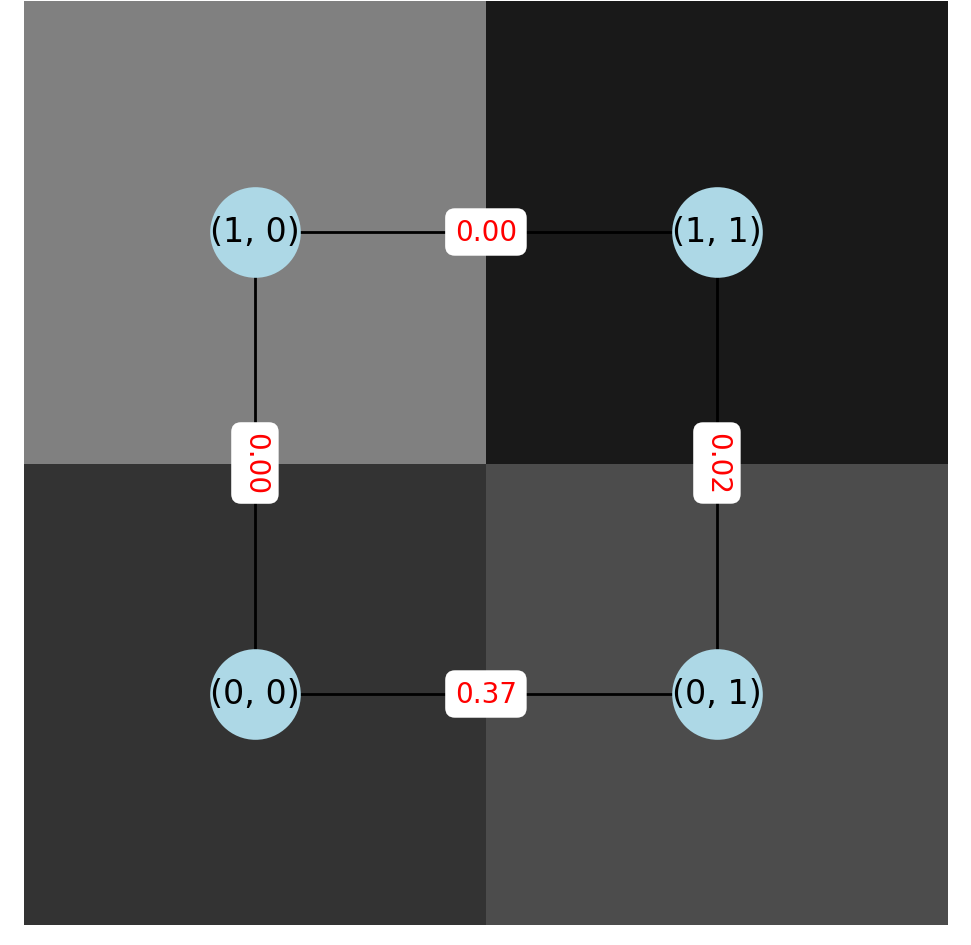} 
  \caption{$2\times2$ grayscale image with an induced graph. Each node corresponds to a pixel and the red edge labels indicate the weights based on similarity and spatial proximity. }
  \label{fig:example}
\end{figure}
\end{example}
We also refer to Fig. \ref{fig:sample} and Fig. \ref{fig:three-inline} for a true image in grayscale and its step-by-step conversion to its graph representation, respectively. Moreover, Fig. \ref{fig:sample2} shows the true grayscale image and magnitude of its graph Laplacian. 
\section{The iterative regularization method assisted by graph Laplacian}\label{sec:IRMGP} 
\begin{figure}
    \centering
    \includegraphics[width=0.47\textwidth]{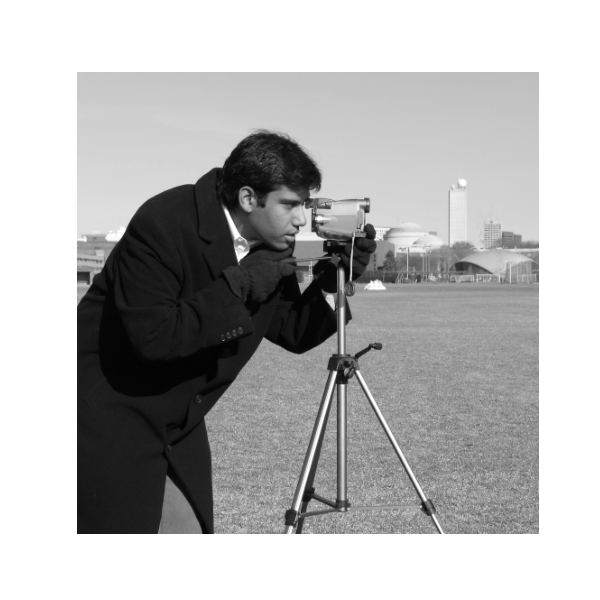} 
    \caption{True image in grayscale.}
    \label{fig:sample}
\end{figure}
In this section, we propose and analyze a novel iterative regularization method that incorporates graph Laplacian-based information, referred to as \texttt{IRMGL+\(\Psi\)}. We conduct a rigorous convergence analysis of the method and examine its regularization properties when combined with a standard a posteriori stopping criterion (\ref{eqn: discrepancy}).
The following baseline assumptions on the forward operator~$A$ are imposed.
\begin{assumption}\label{hypo: bounded}
    $A: \mathcal{D}(A) \subset \mathcal{U} \to \mathcal{V}$ is a bounded linear operator.
\end{assumption}
\begin{assumption}\label{hypo: soln existence}
       There exists $u^\dagger \in \mathcal{B}(u_0, \wp)$ such that $Au^\dagger=v,$ where $\mathcal{B}(u_0, \wp) := \{u \in \mathcal{U} : \|u - u_0\| \leq \wp\}$ denotes a closed ball of radius~$\wp$ centered at~$u_0$.
\end{assumption}
\begin{remark}
\cref{hypo: soln existence} ensures that the exact solution $u^\dagger$ is consistent with the forward model and lies in a small neighborhood of the reconstruction obtained from the exact data, $u_0=\Psi(v)$. This assumption is reasonable for several standard choices of $\Psi$.
For the Tikhonov regularizer 
\[
 \Psi_{\mathrm{Tik}}(v) = \arg\min_{u \in \mathcal{D}(A)}\Bigl\{\tfrac{1}{2}\|Au-v\|^2 + \tfrac{\lambda}{2}\|u\|^2\Bigr\},
\]
the normal equation yields $ \Psi_{\mathrm{Tik}}(v)=(A^*A+\lambda I)^{-1}A^*A\,u^\dagger$. If $A^*A$ is boundedly invertible with smallest singular value $\sigma_{\min}>0$, then
\[
\| \Psi_{\mathrm{Tik}}(v) - u^\dagger\| \le \frac{\lambda}{\sigma_{\min}^2+\lambda}\,\|u^\dagger\| = \mathcal{O}(\lambda).
\]
Hence, for sufficiently small $\lambda$, $u^\dagger$ lies within $\mathcal B(u_0,\wp)$ for an appropriately chosen $\wp$.
Similarly, for the total variation (TV) regularizer 
\[
\Psi_{\mathrm{TV}}(v) = \arg\min_{u \in \mathcal{D}(A)}\Bigl\{\tfrac{1}{2}\|Au-v\|^2 + \lambda\,\mathrm{TV}(u)\Bigr\},
\]
assuming $v=Au^\dagger$ and $\mathrm{TV}(u^\dagger)<\infty$, classical results in variational regularization (see, e.g., \cite{Burger2004,engl1996regularization}) guarantee that $\Psi_{\mathrm{TV}}(v) \to u^\dagger$ as $\lambda\downarrow0$, provided $u^\dagger$ is the minimal-TV solution of $Au=v$. Consequently, for small $\lambda$, $u^\dagger\in\mathcal B(u_0,\wp)$.

In practice, initialization schemes such as FBP, TV and Tikhonov generally provide reconstructions that are quantitatively closer to the true solution \(u^\dagger\), while the adjoint-based initialization \(A^*v\) typically offers a rougher approximation and may therefore require a larger neighborhood radius \(\wp\) to satisfy the theoretical assumption. It is important to emphasize that this assumption is primarily introduced for analytical convenience. In practice, the method demonstrates stable and consistent performance across different initializations, indicating that the assumption holds limited practical relevance.
\end{remark}

\begin{figure}[htbp]
       \includegraphics[width=0.98\textwidth, height=4.5cm]{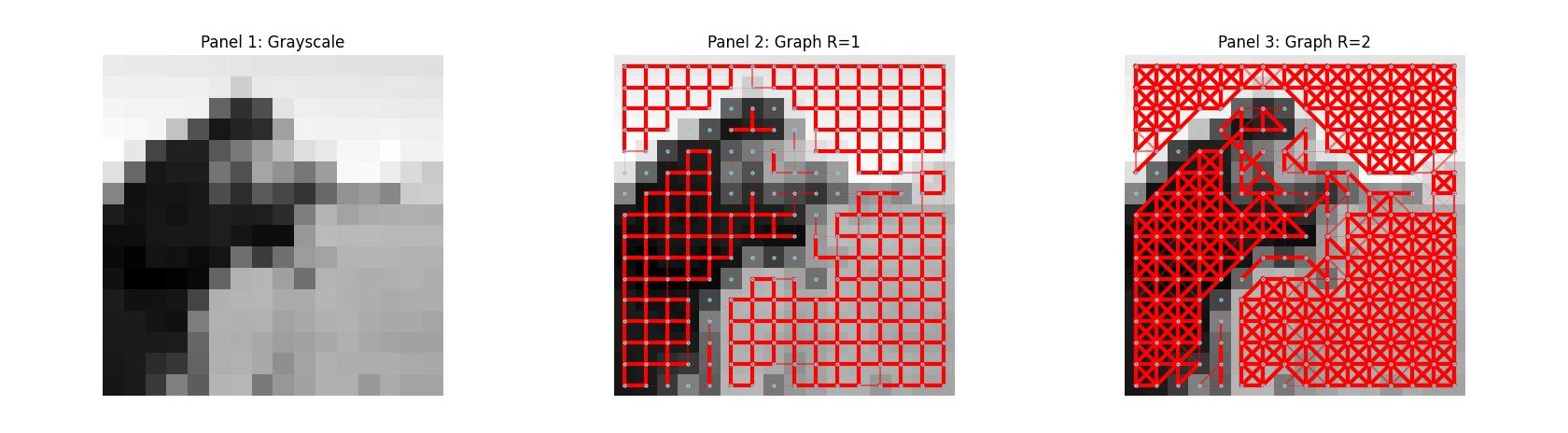}
    \caption{ This figure provides a step-by-step illustration of the process of converting an image $\mathbf{x}$ into a graph representation. On left, a $16 \times 16$ grayscale image is displayed, where the intensity of each pixel is determined by the function $\mathbf{x}$, which encapsulates the underlying image information. In second panel, each pixel is mapped to a graph node (depicted as a sky blue circle), with its location corresponding to its discrete position in the grid $\mathbb{Z}^2$. Edges between nodes are then constructed using the data-dependent weight function $w_{\mathbf{x}}(a, b)$, parameterized by $R = 1$ and $\sigma = 0.005$, which quantifies similarity based on pixel intensities. The strength of these connections is visually represented by the thickness of red edges: thicker edges indicate higher similarity between adjacent pixels, while thinner edges reflect greater dissimilarity. The right panel illustrates the effect of increasing the neighborhood radius to $R = 2$, resulting in a denser connectivity pattern that incorporates a broader local context.}
    \label{fig:three-inline}
\end{figure}

\begin{figure}
    \centering
    \includegraphics[width=1.0\textwidth]{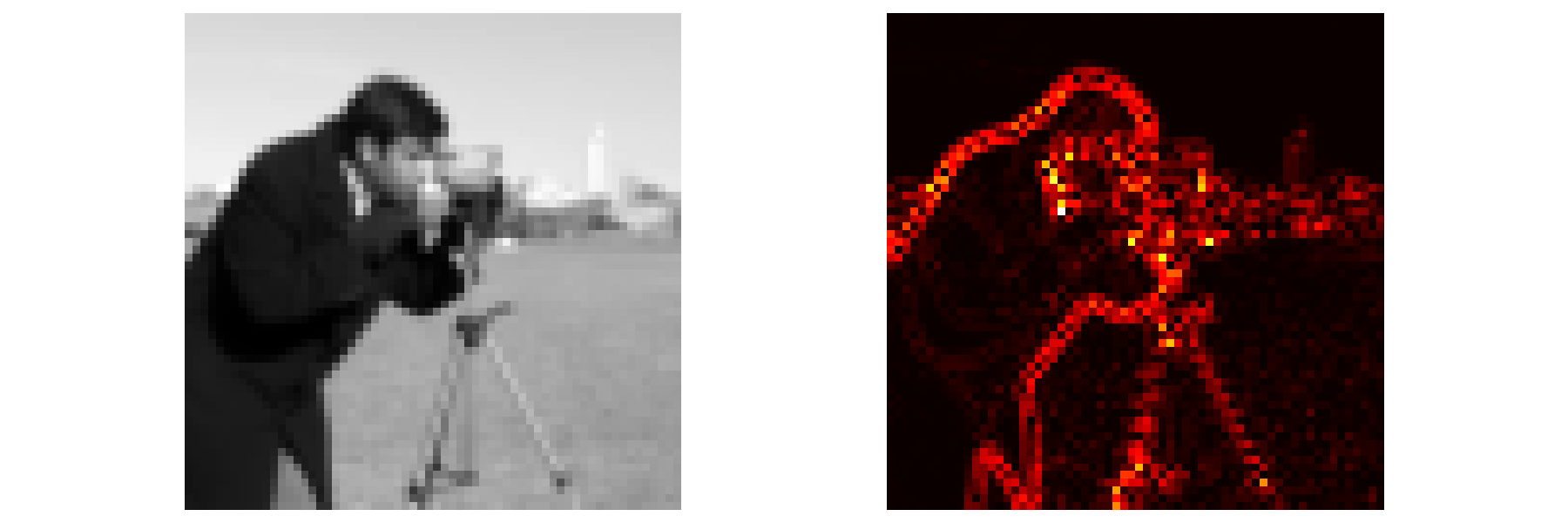} 
   \caption{The left panel shows a $64\times64$ grayscale image. The right panel displays the magnitude of the graph Laplacian, $|\Delta_{\mathbf{x}}\textbf{x}|$, computed using (\ref{Eqn: Graph Laplacian}). This visualization effectively highlights the edge information. The `hot' colormap is used, which maps low values to black and progressively higher values through red, orange, yellow, and white, emphasizing areas with significant local intensity variation.}
    \label{fig:sample2}
\end{figure}

\subsection{The method}\label{subsec: The method}
In this subsection, we present a detailed description of the proposed \texttt{IRMGL+\(\Psi\)} method~\eqref{main iterative schrme}.

\noindent Let us begin by introducing the step size $\alpha_k^\delta$ and the weighted parameter $\beta_k^\delta$ to promote rapid convergence of the method~\eqref{main iterative schrme}. It is natural to select $\alpha_k^\delta$ and $\beta_k^\delta$ adaptively at each iteration in such a way that the iterates remain close to a solution of~(\ref{Model eqn}). Accordingly, we define the parameters 
\begin{equation}\label{alpha}
   0 < \eta \leq  \alpha_k^\delta = 
        \min \left\{ \frac{\eta
    _0\|Au_k^\delta - v^\delta\|^2}{\|A^* (Au_k^\delta -v^\delta)\|^2}, \eta_1 \right\},
\end{equation}
\begin{equation}\label{beta}
   \beta_k^\delta =  \begin{cases}
      \min \left\{ \frac{\nu_0 \|A u_k^\delta - v^\delta\|^2}{\|\Delta_{u_k^\delta} u_k^\delta\|}, \frac{\nu_1}{\|\Delta_{u_k^\delta} u_k^\delta\|}, \nu_2 \right\} , & \text{if } \|\Delta_{u_k^\delta} u_k^\delta\| \neq 0\\ 
      0, &\text{if } \|\Delta_{u_k^\delta} u_k^\delta\| = 0,
    \end{cases} 
\end{equation}
where \(\eta_0, \eta_1, \nu_0, \nu_1\) and \(\nu_2\) are some fixed positive constants. 
\begin{remark}
    Since $A$ is a bounded linear operator, $\|A^*(Au_k^\delta - v^\delta)\| \leq \|A^*\|\|Au_k^\delta - v^\delta\|$. Hence, whenever $ \|A^* (Au_k^\delta -v^\delta)\| \neq 0$, we have
    \[ \frac{
\|Au_k^\delta - v^\delta\|^2}{\|A^* (Au_k^\delta -v^\delta)\|^2} \geq \frac{1}{\|A^*\|^2} = \frac{1}{\|A\|^2}.\] 
Therefore,  it holds
\[\alpha_k^\delta = \min \left\{ \frac{\eta
    _0\|Au_k^\delta - v^\delta\|^2}{\|A^* (Au_k^\delta -v^\delta)\|^2}, \eta_1 \right\} \geq \min \left\{ \frac{\eta
    _0}{\|A\|^2}, \eta_1 \right\} = \underline{\eta} >0.\]
    Thus any fixed choice $0 < \eta \leq \underline{\eta} = \min \left\{ \frac{\eta
    _0}{\|A\|^2}, \eta_1 \right\}$ gives the iteration independent positive lower bound for the admissible step size. In the special case where \( \|A^*(A u_k^\delta - v^\delta)\| = 0 \), the condition  
\begin{equation*}
0 < \eta \leq \alpha_k^\delta = \eta_1
\end{equation*}
still ensures the existence of such a positive lower bound.
\end{remark}
The pseudo code of the iterative scheme~\eqref{main iterative schrme} is outlined in \Cref{alg:buildtree}.
\begin{algorithm}
\caption{\texttt{IRMGL+\(\Psi\)} with early stopping}
\label{alg:buildtree}
\begin{algorithmic}
\STATE{Given: $A, v^\delta, \delta, \Psi, \sigma, R.$}
\STATE{Initialize: $u_0^\delta=\Psi(v^\delta)$}
\WHILE{$\|Au_k^\delta - v^\delta\|> \tau \delta$}
\STATE{Given $u_{k}^\delta, R, \sigma,$ compute $\Delta_{u_k^\delta}$ as defined in Section \ref{sec:prem}.}
\STATE{Update $u_{k+1}^\delta = u_k^\delta - \alpha_k^\delta A^*(Au_k^\delta - v^\delta) - \beta_k^\delta \Delta_{u_k^\delta}u_k^\delta,$}
\STATE{where $\alpha_k^\delta$ and $\beta_k^\delta$ are defined in \eqref{alpha} and \eqref{beta}, respectively.}
\ENDWHILE
\RETURN $u_{k_\delta}^\delta.$
\end{algorithmic}
\end{algorithm}
\begin{remark}
At each \((k+1)\)-th iteration of the Algorithm~\ref{alg:buildtree}, a new graph Laplacian \(\Delta_{u_k^\delta}\) is constructed based on the previous approximate solution \(u_k^\delta\), for any given initial reconstructor \(\Psi\). In other words, the regularization term is adaptively updated at each iteration.
\end{remark}
Next, we introduce the constant  
\begin{equation}\label{assum: On C}
    C := \eta - \frac{\eta_1}{\tau} - \nu_0(\wp + \nu_1) - \eta_0\eta_1,
\end{equation}
and identify a set of parameter choices that ensures \( C > 0 \).  
Assuming \( C > 0 \), we immediately obtain the inequality  
\[
    \eta_1 \geq \eta > \frac{\eta_1}{\tau} + \nu_0(\wp + \nu_1) + \eta_0\eta_1.
\]
This, in turn, implies that  
\[
    1 - \frac{1}{\tau} - \eta_0 > \frac{\nu_0(\wp + \nu_1)}{\eta_1}.
\]
Hence, a convenient strategy for selecting the parameters is as follows. First, choose \(\tau\) such that  
\(
    \tau > \frac{1}{1 - \eta_0},
\)
which guarantees that the left-hand side of the above inequality is positive. Then, select \(\nu_0, \nu_1,\) and \(\eta_1\) so that  
\[
    \nu_0 < \frac{\eta_1\big(1 - \frac{1}{\tau} - \eta_0\big)}{\wp + \nu_1}.
\]
Under these choices, the constant \( C \) remains strictly positive for all admissible parameter values.
\begin{lemma}\label{Lemma: Monotonicity}
    Let \Cref{hypo: bounded} and $\ref{hypo: soln existence}$ hold. Consider \Cref{alg:buildtree} and let $n \leq k_\delta$  be an integer, where $k_\delta$ is the stopping index as defined in $(\ref{eqn: discrepancy})$. Let the constant $C$ as in \cref{assum: On C} be positive.
    Then the following hold: 
    \begin{enumerate}
        \item [(i)] The iterates of \texttt{IRMGL+\(\Psi\)} method $(\ref{main iterative schrme})$ are monotonic  for any $0 \leq k < n,$ i.e.,
        $$\|u_{k+1}^\delta - u^\dagger\| \leq  \|u_{k}^\delta - u^\dagger\|.$$
        \item [(ii)] The iterates $u_k^\delta \in \mathcal{B}(u^\dagger, \wp)$ for all $0 \leq k \leq n$ and  
        \[\sum_{k=0}^{n}  \|Au_k^\delta - v^\delta\|^2 \leq  \frac{1}{2C} \|u_{0}^\delta - u^\dagger\|^2.\] 
        \end{enumerate}
\end{lemma}
\begin{proof}
By \Cref{hypo: soln existence}, we have $u_0 \in \mathcal{B}(u^\dagger, \wp)$. Suppose further that $u_k^\delta \in \mathcal{B}(u^\dagger, \wp)$ for some iteration $k$. This implies that $u_k^\delta \in \mathcal{B}(u_0, 2\wp)$. We define
\[
x_k^\delta := u_k^\delta - u^\dagger \quad \text{and} \quad d_k^\delta := \alpha_k^\delta A^*(Au_k^\delta - v^\delta) + \beta_k^\delta \Delta_{u_k^\delta}u_k^\delta.
\]
With these definitions and in view of the update rule given in (\ref{main iterative schrme}), we obtain
\begin{equation}\label{split eqn}
    \|x_{k+1}^\delta\|^2 - \|x_k^\delta\|^2 = \|d_k^\delta\|^2 - 2\langle x_k^\delta, d_k^\delta \rangle.
\end{equation}
We now proceed to analyze each term on the right-hand side of the \eqref{split eqn}  individually. To begin, we consider the second term and estimate it as
\begin{align}\label{ek, dk}
 \nonumber  - \langle x_k^\delta, d_k^\delta \rangle &= \alpha_k^\delta \langle A(u^\dagger - u_k^\delta), Au_k^\delta - v^\delta \rangle + \beta_k^\delta \langle u^\dagger - u_k^\delta, \Delta_{u_k^\delta}u_k^\delta \rangle \\ \nonumber
   & = - \alpha_k^\delta \|Au_k^\delta - v^\delta\|^2 + \alpha_k^\delta\langle v - v^\delta, Au_k^\delta - v^\delta \rangle +  \beta_k^\delta \langle u^\dagger - u_k^\delta, \Delta_{u_k^\delta}u_k^\delta \rangle \\ \nonumber
   & \leq  - \alpha_k^\delta \|Au_k^\delta - v^\delta\|^2 + \alpha_k^\delta\delta\|Au_k^\delta - v^\delta\| + \beta_k^\delta\wp \| \Delta_{u_k^\delta}u_k^\delta\| \\ 
   & \leq - \left(\eta - \frac{\eta_1}{\tau}\right)\|Au_k^\delta - v^\delta\|^2 + \wp\nu_0 \|Au_k^\delta - v^\delta\|^2,
\end{align}
where we have used (\ref{eqn: delta}), (\ref{eqn: discrepancy}) and \eqref{alpha} in deriving \eqref{ek, dk}. 
Similarly, by using the definitions of $\alpha_k^\delta$ and $ \beta_k^\delta$ with the inequality $(a_1 + a_2)^2 \leq 2(a_1^2 + a_2^2)$ for $a_1, a_2 \in \mathbb{R}$, we can estimate the first term of (\ref{split eqn}) as
\begin{align}\label{dk}
 \nonumber   \|d_k^\delta\|^2 &= \| \alpha_k^\delta A^*(Au_k^\delta - v^\delta) + \beta_k^\delta \Delta_{u_k^\delta}u_k^\delta\|^2 \\ \nonumber
    & \leq 2 \eta_0\eta_1 \|Au_k^\delta - v^\delta\|^2 + 2\nu_0 \nu_1\|Au_k^\delta - v^\delta\|^2 \\ 
    & \leq 2(\eta_0\eta_1+ \nu_0 \nu_1)\|Au_k^\delta - v^\delta\|^2. 
\end{align}
 Further, by using (\ref{ek, dk}) and (\ref{dk}) in (\ref{split eqn}), we get
\begin{equation}\label{Concluding eqn of first Lemma}
     \|x_{k+1}^\delta\|^2 - \|x_k^\delta\|^2 \leq -2\left( \eta- \frac{\eta_1}{\tau} -\nu_0(\wp + \nu_1) - \eta_0\eta_1\right)\|Au_k^\delta - v^\delta\|^2.
\end{equation}
Finally, by substituting (\ref{assum: On C}) in (\ref{Concluding eqn of first Lemma}), we obtain the required monotonicity for $0 \leq k <n$.
Next, we discuss the proof of assertion $(ii)$. The inequality (\ref{Concluding eqn of first Lemma}) guarantees that
\[
\|u_{k+1}^\delta - u^\dagger\| \leq \|u_{k}^\delta - u^\dagger\| \leq \wp,
\]
which implies that $u_{k+1}^\delta \in \mathcal{B}(u^\dagger, \wp)$. By summing the inequality (\ref{Concluding eqn of first Lemma}) from $k = 0$ to $k = n$, we obtain
\begin{equation}\label{Concluding eqn of first Lemma in C}
   \sum_{k=0}^{n} \left[\|x_{k+1}^\delta\|^2 - \|x_k^\delta\|^2\right] \leq -2C  \sum_{k=0}^{n}\|Au_k^\delta - v^\delta\|^2.
\end{equation}
Finally, it follows from the last inequality that
\begin{equation*}
 \sum_{k=0}^{n}\|Au_k^\delta - v^\delta\|^2 \leq \frac{1}{2C} \|x_{0}^\delta\|^2 = \frac{1}{2C}\|u_0^\delta - u^\dagger\|^2.
\end{equation*}
This establishes the desired result.
\end{proof}

\noindent In the following result, we show that \texttt{IRMGL+\(\Psi\)} method terminates in a finite number of steps, i.e., Algorithm \ref{alg:buildtree} is well defined. The proof leverages Lemma~\ref{Lemma: Monotonicity} to substantiate this conclusion.

\begin{theorem}
    Let all the assumptions of \Cref{Lemma: Monotonicity} be satisfied. Let \Cref{alg:buildtree} be initiated with given $u_0^\delta$. Then the algorithm must terminates in finitely many steps, i.e., there exist a finite integer $k_\delta$ such that 
    \begin{equation*}
    \|Au_{k_\delta}^\delta - v^\delta\| \leq \tau \delta < \|Au_k^\delta - v^\delta\|, \quad 0 \leq k < k_\delta.
    \end{equation*}
\end{theorem}
\begin{proof}
Suppose \( n \geq 0 \) is an integer such that \( \|A u_k^\delta - v^\delta\| > \tau \delta \) holds for all \( k = 0, 1, \dots, n \). Then, be engaging \eqref{Concluding eqn of first Lemma in C}, we obtain
\begin{align*}
     2C\sum_{k=0}^{k=n}  \|Au_k^\delta - v^\delta\|^2 &\leq \sum_{k=0}^{k = n}\left[\|x_{k}^\delta\|^2 - \|x_{k+1}^\delta\|^2\right],\\
     & = \|x_0^\delta\|^2 - \|x_{n+1}^\delta\|^2 \leq \|x_0^\delta\|^2.
\end{align*}
By virtue of \eqref{eqn: discrepancy}, we deduce that
\begin{equation}\label{eqn: thm 3.5 conclusion}
  2(n+1)C\tau^2\delta^2\leq  2C\sum_{k=0}^{k=n}  \|Au_k^\delta - v^\delta\|^2 \leq \|x_0^\delta\|^2 < \infty.
\end{equation}
If there exists no finite integer \( k_\delta \) such that condition (\ref{eqn: discrepancy}) is satisfied, then by taking the limit as \( n \to \infty \) in \eqref{eqn: thm 3.5 conclusion}, we get a contradiction. Therefore,  Algorithm \ref{alg:buildtree} is guaranteed to terminate after a finite number of iterations.
\end{proof}
\subsection{Convergence for exact data}\label{Subsec: Convergence for exact data}
In this subsection, we investigate the counterpart of Algorithm~\ref{alg:buildtree} in the form of \Cref{alg:exactdata} that employs the exact data \( v \) instead of $v^\delta$.
\begin{algorithm}
\caption{\texttt{IRMGL+\(\Psi\)} with exact data}
\label{alg:exactdata}
\begin{algorithmic}
\STATE{Given: $A, v, \Psi, \sigma, R.$}
\STATE{Initialize: $u_0=\Psi(v)$}
\WHILE{$k \geq0$}
\STATE{Given $u_{k}, R, \sigma,$ compute $\Delta_{u_k}$ as defined in Section \ref{sec:prem}.}
\STATE{Update $u_{k+1} = u_k - \alpha_k A^*(Au_k - v) - \beta_k \Delta_{u_k}u_k,$}
\STATE{where $\alpha_k$ and $\beta_k$ are defined as  \begin{equation*}
     \alpha_k = \min \left\{ \frac{\eta
    _0\|Au_k - v\|^2}{\|A^* (Au_k -v)\|^2}, \eta_1 \right\} \quad \text{when } Au_k \neq v,
\end{equation*}
\begin{equation*}
   \beta_k =  \begin{cases}
      \min \left\{ \frac{\nu_0 \|A u_k - v\|^2}{\|\Delta_{u_k} u_k\|}, \frac{\nu_1}{\|\Delta_{u_k} u_k\|}, \nu_2 \right\} , & \text{if } \|\Delta_{u_k} u_k\| \neq 0\\ 
      0, &\text{if } \|\Delta_{u_k} u_k\| = 0
    \end{cases} 
\end{equation*}}
\ENDWHILE
\end{algorithmic}
\end{algorithm}
To establish the convergence of our method for exact data, we demonstrate that the sequence \(\{u_k\}_{k \geq 0} \) generated by Algorithm~\ref{alg:exactdata} forms a Cauchy sequence. 
We begin by presenting a preliminary result for the sequence \( \{u_k\}_{k \geq 0} \).
 
\begin{lemma}\label{lemma: monotonicity for exact data}
    Let \Cref{hypo: bounded} and  \Cref{hypo: soln existence} hold. Let $\{u_k\}_{k \geq 0}$ be the sequence of iterates generated by the \Cref{alg:exactdata}. Then $u_k \in \mathcal{B}(u^\dagger, \wp)$ for all integers $k \geq 0$ and for the solution $u^\dagger$ of $(\ref{Model eqn})$, there holds
 \begin{equation}
     \|u_{k+1} - u^\dagger\|^2 - \|u_k - u^\dagger\|^2 \leq -2C_0 \|Au_k - v\|^2 \quad \forall \; k \geq0,
 \end{equation}
 where $$C_0: = \eta  -\nu_0(\wp + \nu_1) - \eta_0\eta_1 > 0.$$ Also, the sequence $\{\|u_k - u^\dagger\|\}_{k \geq 0}$ is monotonically decreasing and 
 \begin{equation}\label{sum for non noisy case}
     \sum_{k=0} ^\infty \|Au_k - v\|^2 \leq \frac{1}{2C_0} \|u_0 - u^\dagger\|^2 < \infty.
 \end{equation}
 This means that $\|Au_k - v\| \to 0$ as $k \to \infty$ i.e.,  the residual norm converges to zero. 
\end{lemma}
\begin{proof}
   The result follows directly by analogy with the proof of Lemma \ref{Lemma: Monotonicity}, and the positivity of $C_0$ can be guaranteed in the same manner as that of $C$.
\end{proof}
  \noindent To this end, we are ready to present the convergence result in the case of $\delta=0$.
\begin{theorem}\label{Convergence for exact data}
   Under the assumptions of \Cref{lemma: monotonicity for exact data}, the iterates $\{u_k\}_{k \geq 0}$ generated by \Cref{alg:exactdata} converges to $u^\dagger$, which is a solution of $(\ref{Model eqn})$.
\end{theorem}
\begin{proof}
    For $\delta =0,$ we define $x_k:= u_k - u^\dagger$ and $r_k:=Au_k -v$. Given $l \geq k$, we choose an integer $m$ with $l \geq m \geq k$ such that
    \begin{equation}\label{rn rp ineq}
        \|r_m\| \leq \| r_n\|, \quad \forall \hspace{1mm} k \leq n \leq l.
    \end{equation}
    By triangular inequality, we observe that 
\begin{equation}\label{xk triangular}
    \|x_l - x_k\| \leq \|x_l - x_m\| + \|x_m - x_k\|,
\end{equation}    
where 
\begin{align}
    \|x_l - x_m\|^2 &= 2\langle x_m - x_l, x_m \rangle + \|x_l\|^2 - \|x_m\|^2, \label{xl_xn} \\
    \|x_m - x_k\|^2 &= 2\langle x_k - x_m, x_m \rangle + \|x_m\|^2 - \|x_k\|^2. \label{xn_xm}
\end{align}
 From Lemma \ref{lemma: monotonicity for exact data} it is clear that $\{\|x_k\|\}_{k \geq 0}$ is a monotonically decreasing sequence bounded below by $0.$ Consequently, we suppose that \(\lim_{k \to \infty}\|x_k\| = \theta \geq 0.\) Hence, as a result, we write
\begin{align}
     \lim_{k \to \infty }\|x_l - x_m\|^2 &= 2 \lim_{k \to \infty }\langle x_m - x_l, x_m \rangle + \theta^2 - \theta^2, \label{xl xm limit}\\
     \lim_{k \to \infty }\|x_m - x_k\|^2 &= 2 \lim_{k \to \infty }\langle x_k - x_m, x_m \rangle + \theta^2 - \theta^2.\label{xm xk limit}
\end{align} 
Our aim is to demonstrate that $\{x_k\}_{k \geq 0}$ is a Cauchy sequence. To this end, we claim that $\langle x_m - x_l, x_m \rangle \to 0$ as $k \to \infty$. For proving the claim, it can be observe that  
\begin{align}
 \nonumber   |\langle x_m - x_l, x_m\rangle| &= |\langle u_m - u_l, x_m\rangle| \\ \nonumber
    & \leq \left| \left\langle \sum_{j=m}^{l-1} \alpha_j A^*(Au_j - v) + \beta_j \Delta_{u_j}u_j, u_m - u^\dagger \right\rangle \right| \\ \nonumber
     & \leq \sum_{j=m}^{l-1}\left| \left\langle  \alpha_j A^*(Au_j- v), u_m - u^\dagger \right\rangle \right| + \sum_{j=m}^{l-1}\left| \left\langle \beta_j \Delta_{u_j}u_j, u_m - u^\dagger \right\rangle \right| \\ \nonumber
     & \leq \eta_1 \sum_{j=m}^{l-1}\left| \left\langle   Au_j- v, Au_m - v \right\rangle \right| + \sum_{j=m}^{l-1}\beta_j \left| \left\langle  \Delta_{u_j}u_j, u_m - u^\dagger \right\rangle \right|, 
\end{align}
where, we have utilized the bound on $\alpha_j$ and the relation $Au^\dagger =v$ in deriving the last inequality.
Furthermore, by incorporating the definition of $\beta_k$ together with inequality~(\ref{rn rp ineq}), and using the fact that $u_k \in \mathcal{B}(u^\dagger, \wp)$ for all $k \geq 0$, the above inequality can be reformulated as 
\begin{align}\label{x_n x_l in terms of residual}
   \nonumber   |\langle x_m - x_l, x_m\rangle| & \leq \eta_1 \sum_{j=m}^{l-1}\|Au_j - v\|^2 + \wp \nu_0 \sum_{j=m}^{l-1}\|Au_j - v\|^2 \\ 
   &  \leq (\eta_1 + \wp\nu_0)\sum_{j=m}^{l-1}\|Au_j - v\|^2.
\end{align}
Reasoning in the same way for $|\langle x_k - x_m, x_m \rangle|$, we may write
\begin{equation}
 |\langle x_k - x_m, x_m \rangle| \leq    (\eta_1 + \wp \nu_0)\sum_{j=k}^{l-1}\|Au_j - v\|^2.
\end{equation}
From these estimates, it follows that \( |\langle x_m - x_l, x_m \rangle| \) and \( |\langle x_k - x_m, x_m \rangle| \) tend to zero as \( k \to \infty \), owing to \eqref{sum for non noisy case}. Consequently, from (\ref{xl xm limit}) and \eqref{xm xk limit} it follows that
\begin{align*}
     \lim_{k \to \infty }\|x_l - x_m\|^2 &= 2 \lim_{k \to \infty }\langle x_m - x_l, x_m \rangle =0, \\
     \lim_{k \to \infty }\|x_m - x_k\|^2 &= 2 \lim_{k \to \infty }\langle x_k - x_m, x_m \rangle =0.
\end{align*}
Hence, from \eqref{xk triangular}, \eqref{xl_xn} and \eqref{xn_xm}, it follows that the sequence \( \{x_k\}_{k \geq 0} \) is Cauchy. By the definition of \( x_k \), this immediately implies that the sequence \( \{u_k\}_{k \geq 0} \) is also Cauchy and therefore converges to some limit \( \hat{u} \). Furthermore, since the residuals \( Au_k - v \) tend to zero as \( k \to \infty \), the limit \( \hat{u} \) satisfies \eqref{Model eqn}, which concludes that $\hat{u} = u^\dagger$. Consequently, the iterates  converge to a solution of \eqref{Model eqn}.
\end{proof}
\subsection{Stability analysis}\label{subsec: stability}
In this subsection, we establish the stability of the \texttt{IRMGL+\(\Psi\)} method. For that, we require the following assumption. 
 \begin{assumption}\label{hypothesis stability}
     Let $v^\delta
     $ be the sequence of noisy data such that \(v^\delta \to v\) as \(\delta \to 0\). Then the reconstructor $\Psi$ satisfies
     \begin{equation*}
         \|\Psi(v^\delta) - \Psi(v)\| \to 0 \quad \text{as} \quad \delta \to 0.
     \end{equation*}
 \end{assumption}
It is worth noting that the above assumption is relatively mild and is satisfied by a broad class of reconstruction operators. A straightforward example is to consider \( \Psi \) as a (locally) Lipchitz continuous operator. That is, if \( \Psi \) satisfies
\[
\|\Psi(v^\delta) - \Psi(v)\| \leq K \|v^\delta - v\|,
\]
for some constant \( K > 0 \), then the assumption is clearly fulfilled. For a comprehensive discussion of the different reconstructors \(\Psi\) satisfying \cref{hypothesis stability}, the reader is referred to \cref{subsec:initial_rec}.


\noindent
The main difficulty in establishing stability results stems from the regularization term \( \mathcal{R}(u, v^\delta) \), which explicitly depends on the observed data \( v^\delta \).
Consequently, conventional analytical techniques cannot be applied in a straightforward manner and hence we propose a lemma that plays a crucial role in the stability analysis. It is worth noting that a similar result was established in \cite{bianchi2025data} for the variational setting, however, the analysis presented here is fundamentally different from that work.

\begin{lemma}\label{lemma: contraction of Delta}
  For the iterates \(u_k^\delta\) and $u_k$ generated by \Cref{alg:buildtree} and \Cref{alg:exactdata} respectively, it holds that 
  \begin{equation*}
      \|\Delta_{u_k^\delta}u_k - \Delta_{u_k}u_k\|_2 \leq \mathcal{H} \|u_k\|_2\|u_k^\delta - u_k\|_2,
  \end{equation*}
    where the constant \[\mathcal{H} =2 L_h \left( \max_{a,b}{g(a, b)}\sqrt{N}  +  \max_{a}\max_{b}{g(a, b)}\right),\] with \(g(a,b)\) the edge weight function specified in \eqref{eqn: edge weight fun}.
\end{lemma}
\begin{proof}
From triangular inequality and submultiplicativity of the operator norm, we have 
\[\|\Delta_{u_k^\delta}u_k - \Delta_{u_k}u_k\|_2 \leq \|\Delta_{u_k^\delta} - \Delta_{u_k}\|_s \|u_k\|_2,\]
where $\|\cdot\|_s$ denotes the spectral norm (i.e., matrix operator norm induced by the Euclidean vector norm). We recall that, for any vector $u$, the graph Laplacian is given by 
\[
\Delta_{u} = D_{u} - W_{u},
\]
where the weight matrix \( W_{u} \) and the diagonal degree matrix \(D_{u}\) are same as defined in subsection~\ref{subsec: Graph Theory}. Using this result, it can be noted that
\begin{equation}\label{eqn: delta split in D and W}
  \| \Delta_{u_k^\delta} - \Delta_{u_k} \|_s \leq \| D_{u_k^\delta} - D_{u_k} \|_s + \| W_{u_k^\delta} - W_{u_k} \|_s.  
\end{equation}
\noindent
Next, we individually estimate both the terms of right hand side of the above inequality, starting with the second term. For that, the entry-wise difference is 
\[
|w_{u_k}(a, b) - w_{u_k^\delta}(a, b)|= g(a, b)|h_{u_k}(a, b)-h_{u_k^\delta}(a, b) |,
\]
where \( h_{u} \) is a Lipchitz continuous Gaussian kernel function for any vector $u$ with Lipchitz constant $L_h$. Hence, we have that
\begin{align*}
    |w_{u_k}(a, b) - w_{u_k^\delta}(a, b)| &= g(a, b)L_h \left||u_k(a) -u_k(b)| - |{u_k^\delta}(a) -  {u_k^\delta}(b)| \right| \\
    & \leq g(a, b)L_h  \left|(u_k(a) -u_k(b)) - ({u_k^\delta}(a) -  {u_k^\delta}(b))\right| \\
    & \leq g(a, b)L_h  \left|(u_k(a) -u_k^\delta(a)) - (u_k(b) -  {u_k^\delta}(b))\right|\\
    &\leq g(a, b)L_h  \left(\left|u_k(a) -u_k^\delta(a)\right| + \left|u_k(b) -  {u_k^\delta}(b)\right|\right).
\end{align*}
Thus, we arrive at
\[\| W_{u_k^\delta} - W_{u_k} \|_F^2 \leq (\max_{a,b}{g(a, b)}L_h)^2 \sum_{a,b \in S} \left(|u_k(a) -u_k^\delta(a)| + |u_k(b) -  {u_k^\delta}(b)|\right)^2,\] 
where $\|\cdot\|_F$ denotes the Frobenius norm and $S$ denotes the set associated with the edge-weight function $g(a, b).$  Further, from \cref{eqn:degree}, we have
\begin{align*}
&\|W_{u_k^\delta} - W_{u_k}\|_F^2\\
&= 2(\max_{a,b}{g(a, b)}L_h)^2 \Bigl(
\sum_{a\in S} \deg(a)\,|u_k(a) -u_k^\delta(a)|^2 
+ \sum_{b\in S} \deg(b)\,|u_k(b) -  {u_k^\delta}(b)|^2
\Bigr) \\
&\le 4(\max_{a,b}{g(a, b)}L_h)^2 N \sum_{a\in S} |u_k(a) -u_k^\delta(a)|^2 
= 4(\max_{a,b}{g(a, b)}L_h)^2 N\,\|u_k^\delta - u_k\|_2^2.
\end{align*}
After solving the above estimate and using the fact that \(\|\cdot\|_s \leq \|\cdot\|_F\), we get
\begin{equation}\label{eqn: Wuk estimate}
\| W_{u_k^\delta} - W_{u_k} \|_s \leq  \| W_{u_k^\delta} - W_{u_k} \|_F  \leq 2 \max_{a,b}{g(a, b)}L_h\sqrt{N} \|u_k^\delta - u_k\|_2.
\end{equation}
Similarly, we estimate the first term of \eqref{eqn: delta split in D and W} in which each diagonal entry satisfies
\begin{align*}
 |d_{u_k}(a)-d_{u_k^\delta}(a)| &= \left| \sum_{b \in S} (w_{u_k}(a, b) - w_{u_k^\delta}(a, b)) \right| \leq \sum_{b \in S} |w_{u_k}(a, b) - w_{u_k^\delta}(a, b)| \\
 &\leq \max_{b}{g(a, b)}L_h \sum_{b \in S}  \left(|u_k(a) -u_k^\delta(a)| + |u_k(b) -  {u_k^\delta}(b)| \right) \\
 & \leq 2 \max_{b}{g(a, b)}L_h   \|u_k -u_k^\delta\|_2.
\end{align*}
This yields that
\begin{equation}\label{eqn: Duk estimate}
    \| D_{u_k^\delta} - D_{u_k} \|_s = \max_{a} |d_{u_k}(a)-d_{u_k^\delta}(a)| \leq 2 \max_{a}\max_{b}{g(a, b)}L_h  
    \|u_k -u_k^\delta\|_2. 
\end{equation}
Plugging (\ref{eqn: Wuk estimate}) and (\ref{eqn: Duk estimate}) in (\ref{eqn: delta split in D and W}) to arrive at
\begin{equation}\label{eqn:last_eqn}
    \|  \Delta_{u_k^\delta} - \Delta_{u_k} \|_s \leq \mathcal{H}  \|u_k^\delta - u_k\|_2,
\end{equation}
where $\mathcal{H}$ is the same constant as defined in the statement. Thus, the result holds.
\end{proof}
\noindent
Next, we show that  the Lemma~\ref{lemma: contraction of Delta}, in conjunction with \Cref{hypothesis stability}, establishes the stability required to link Algorithm~\ref{alg:buildtree} with Algorithm~\ref{alg:exactdata}. This is formalized in the following result.

 \begin{lemma}\label{lemma: stability}
 Let Assumptions~$\ref{hypo: bounded}, \ref{hypo: soln existence}$ and $\ref{hypothesis stability}$ hold. Then for each fixed integer $k \geq 0$, there holds 
 \begin{equation}
    u_k^\delta \to u_k \quad \text{as} \quad \delta \to 0,
 \end{equation} 
 where $u_k^\delta$ and $u_k$ are same as in \Cref{alg:buildtree} and \Cref{alg:exactdata}.
 \end{lemma}
\begin{proof}
   We establish the result by induction. For the base case \( k = 0 \), note that \( u_0^\delta = \Psi(v^\delta) \) and \( u_0 = \Psi(v) \). Therefore, the claim holds trivially by \Cref{hypothesis stability}. Next, let us assume that the statement holds for all \( 0 \leq j \leq k \). We aim to show that it also holds for \( j = k+1 \), i.e., \( u_{k+1}^\delta \to u_{k+1} \) as \( \delta \to 0 \). 
    From (\ref{main iterative schrme}), we have 
    \begin{align}\label{x_k+1 to x_k}
    \nonumber u_{k+1}^\delta - u_{k+1} &= (u_k^\delta - u_k) -  \left( \alpha_k^\delta A^*(Au_k^\delta - v^\delta) - \alpha_k A^*(Au_k - v) \right)\\
     & \quad -(\beta_k^\delta \Delta_{u_k^\delta}u_k^\delta - \beta_k \Delta_{u_k}u_k).
    \end{align}
We begin by establishing that  
\[
\alpha_k^\delta A^*(Au_k^\delta - v^\delta) \to \alpha_k A^*(Au_k - v) \quad \text{as } \delta \to 0.
\]
To this end, we utilize that \( u_k^\delta \to u_k \) and \( v^\delta \to v \) as \( \delta \to 0 \). When $Au_k - v = 0$, using $0 \leq \alpha_k^\delta \leq \eta_1$, it follows that
\[
   \left\| \alpha_k^\delta A^*(Au_k^\delta - v^\delta) - \alpha_k A^*(Au_k - v) \right\|
   = \left\| \alpha_k^\delta A^*(Au_k^\delta - v^\delta)\right\|
   \leq \eta_1 \left\| A^*(Au_k^\delta - v^\delta)\right\|.
\]
Since $Au_k^\delta \to Au_k$ and $v^\delta \to v$ as $\delta \to 0$, we obtain
\[
   \eta_1\left\|  A^*(Au_k^\delta - v^\delta)\right\| 
   \to \eta_1 \left\| A^*(Au_k - v)\right\| = 0.
\]
On the other hand, when $Au_k - v \neq 0$, we have
\[
   \langle A^*(Au_k - v),\, u_k - u^\dagger \rangle
   = \langle Au_k - v,\, Au_k - v \rangle
   = \|Au_k - v\|^2 > 0,
\]
which implies that $A^*(Au_k - v) \neq 0$ and we get 
\[
\alpha_k^\delta = \min \left\{ \frac{\eta_0 \|Au_k^\delta - v^\delta\|^2}{\|A^*(Au_k^\delta - v^\delta)\|^2}, \eta_1 \right\}
\to
\min \left\{ \frac{\eta_0 \|Au_k - v\|^2}{\|A^*(Au_k - v)\|^2}, \eta_1 \right\} = \alpha_k.
\]
 Consequently, we obtain
\begin{align}\label{Eqn: (3.26)}
\left\| \alpha_k^\delta A^*(Au_k^\delta - v^\delta) - \alpha_k A^*(Au_k - v) \right\|
& \leq \alpha_k^\delta \left\| A^*(Au_k^\delta - v^\delta) - A^*(Au_k - v) \right\| \nonumber \\
& \quad + \left| \alpha_k^\delta - \alpha_k \right| \left\| A^*(Au_k - v) \right\| \to 0 \text{ as }  \delta \to 0.
\end{align}
In a similar vein, assume that \( \|\Delta_{u_k^\delta} u_k^\delta\| \neq 0 \). Then, by the definitions of \( \beta_k^\delta \) and \( \beta_k \), we obtain
\[
\beta_k^\delta = 
\min \left\{ \frac{\nu_0 \|A u_k^\delta - v^\delta\|^2}{\|\Delta_{u_k^\delta} u_k^\delta\|}, \frac{\nu_1}{\|\Delta_{u_k^\delta} u_k^\delta\|}, \nu_2 \right\} \to 
\min \left\{ \frac{\nu_0 \|A u_k - v\|^2}{\|\Delta_{u_k} u_k\|}, \frac{\nu_1}{\|\Delta_{u_k} u_k\|}, \nu_2 \right\}
\]
as  $\delta \to 0.$ This convergence follows from the induction hypothesis and Lemma~\ref{lemma: contraction of Delta}, which ensures that\( \|\Delta_{u_k^\delta} u_k^\delta\| \to \|\Delta_{u_k} u_k\| \) as \( \delta \to 0 \). Thus, we have
\begin{align} \label{eqn: (3.27)}
  \|\beta_k^\delta \Delta_{u_k^\delta}u_k^\delta &- \beta_k \Delta_{u_k}u_k\| \leq  \|\beta_k^\delta \Delta_{u_k^\delta}u_k^\delta - \beta_k^\delta \Delta_{u_k^\delta}u_k\| + \|\beta_k^\delta \Delta_{u_k^\delta}u_k - \beta_k \Delta_{u_k}u_k\| \nonumber \\ 
&\quad \leq  \|\beta_k^\delta \Delta_{u_k^\delta}(u_k^\delta - u_k)\| +\left(|\beta_k^\delta - \beta_k|\| \Delta_{u_k^\delta}\| +  \beta_k \|\Delta_{u_k^\delta} - \Delta_{u_k}\|\right)\|u_k\|  \nonumber \\ 
&\quad \leq  \|\beta_k^\delta \Delta_{u_k^\delta}\|\|u_k^\delta - u_k\| +\left(|\beta_k^\delta - \beta_k|\| \Delta_{u_k^\delta}\| +  \beta_k \|\Delta_{u_k^\delta} - \Delta_{u_k}\|\right)\|u_k\|,
\end{align} 
where the  boundedness of $\|\Delta_{u_k^\delta}\|$ follows directly from \cref{eqn:last_eqn} by setting $u_k = 0$, which yields $\| \Delta_{u_k^\delta}\| \leq \| \Delta_0\| + \mathcal{H}  \|u_k^\delta\|.$
Hence, applying Lemma~\ref{lemma: contraction of Delta} to the above expression together with the induction hypothesis, we obtain
\[
\beta_k^\delta \Delta_{u_k^\delta} u_k^\delta \to \beta_k \Delta_{u_k} u_k \quad \text{as } \delta \to 0.
\]
Consequently, it follows from \eqref{x_k+1 to x_k}, \eqref{Eqn: (3.26)}, \eqref{eqn: (3.27)} and the induction hypothesis that
\(
u_{k+1}^\delta \to u_{k+1} \text{ as } \delta \to 0.
\) This completes the proof.
\end{proof}
\subsection{Convergence for noisy data}\label{subsec: nonnoisy} In this subsection, our goal is to establish the regularization property of \texttt{IRMGL+\(\Psi\)} method. It is worth emphasizing that, in general, convergence of the method with noisy data cannot be guaranteed, as the perturbed data \( v^\delta \) may not belong to the range of the operator \( A \), i.e., \( v^\delta \notin \mathcal{R}(A) \). 
\begin{theorem}\label{thm: non noisy}
   Under the assumptions of \Cref{lemma: stability}, if $k_\delta$ is the stopping index as defined in \eqref{eqn: discrepancy} for the method defined in  \Cref{alg:buildtree}, then there exist a solution $u^\dagger$ of $(\ref{Model eqn})$ such that 
    \[\lim_{\delta \to 0} \|u_{k_\delta}^\delta - u^\dagger\| = 0.\]
    \begin{proof}
        Let $u^\dagger$ be the solution of (\ref{Model eqn}) determined in Theorem \ref{Convergence for exact data} such that \(\|u_k - u^\dagger\| \to 0\) as \(\delta \to 0\), where $\{u_k\}_{k \geq 0}$ denotes the sequence of iterates determined by the Algorithm~\ref{alg:exactdata} using the exact data. We assume that there is a sequence $\{\delta_n\}$ such that $\delta_n \to 0$ as $n \to \infty,$ and let $v^{\delta_n}$ be the associated noisy data sequence. The corresponding stopping index for each pair $(\delta_n, v^{\delta_n}),$ as established by (\ref{eqn: discrepancy}), is represented by $k_n:=k_{\delta_n}$.
      We shall examine two scenarios in order to demonstrate the desired outcome.
      
      \textit{Case-I:} Consider that for some finite integer $\Bar{k}$, $k_n \to \Bar{k}$ as $n \to \infty$. Then, for all large $n$, we may assume that $k_n = \Bar{k}$. Consequently, using the definition of $k_n$, we obtain
      \begin{equation}\label{discrepancy in last result}
       \|Au_{\Bar{k}}^{\delta_n} - v^{\delta_n}\| \leq \tau \delta_n.   
      \end{equation}
       Taking the limit as \( n \to \infty \) in \eqref{discrepancy in last result} and applying Lemma \ref{lemma: stability}, we obtain \( A u_{\bar{k}} = v \). Then, by invoking Theorem \ref{Convergence for exact data}, it follows that \( u_{\bar{k}} = u^\dagger \). Consequently, we conclude that \( u_{k_n}^{\delta_n} \to u^\dagger \) as \( n \to \infty \).

    \textit{Case-II:} Next, assume that there exist a sequence $\{\delta_n\}$ with  $\delta_n \to 0$ such that $k_n:= \hat{k}_{\delta_n} \to \infty$ as $n \to \infty$, and let $v^{\delta_n}$ be the corresponding sequence of noisy data. Let $k$ be sufficiently large fixed integer. 
        From Lemma~\ref{lemma: monotonicity for exact data}, it follows that there exist an $\varepsilon >0$ such that $\|u_k - u^\dagger\| \leq \frac{\varepsilon}{2}$. Since for fixed $k,$ using Lemma~\ref{lemma: stability}, there is a $n=(k, \varepsilon)$ such that $\|u_k^{\delta_n} - u_k\|\leq \frac{\varepsilon}{2}$ for all $n > n(k, \varepsilon).$
        Then, for sufficiently large $n$, such that $k_n > k$, we conclude from Lemma~\ref{Lemma: Monotonicity} that 
        \begin{align*}
            \|u_{k_n}^{\delta_n}-u^\dagger\| &\leq \|u_{k_n -1}^{\delta_n} - u^\dagger\| \leq \cdots \leq \|u_k^{\delta_n} - u^\dagger\|\\
            & \leq \|u_{k}^{\delta_n} - u_k\| + \|u_k - u^\dagger\| \leq \varepsilon.
        \end{align*}
      Consequently, the outcome follows.
    \end{proof}
\end{theorem}
\section{Numerical experiments and discussion}
\label{sec: numerical}
In this section, we first provide a comparative analysis with existing methods to demonstrate the efficiency, and robustness of our approach. Subsequently,  we present the numerical simulations for two-dimensional (2D) computed tomography (CT) and image deblurring applications. We describe the experimental setup, problem setting, and a number of numerical tests that were performed to assess the effectiveness of the \texttt{IRMGL+\(\Psi\)} method.
\subsection{Comparison Analysis}\label{appendix}
To provide a comprehensive comparison with our proposed \texttt{IRMGL+\(\Psi\)} scheme, this section details the formulations of the \texttt{GraphLa+\(\Psi\)}~\cite{bianchi2025data} and \texttt{it-graphLa\(\Psi\)}~\cite{bianchi2024improved} methods. Both approaches utilize graph Laplacian regularization within a variational framework, but they differ in their handling of the updates of regularization term.

\subsubsection{The \texttt{GraphLa+\(\Psi\)} Method}
\label{app:graphla_psi}
The method, introduced in \cite{bianchi2025data}, is a variational regularization approach for solving inverse problems. It incorporates a graph Laplacian as a static regularization term, assuming a fixed underlying data structure. The core idea is to find the approximated solution  $u_\beta$ that minimizes a functional combining a data fidelity term and a Tikhonov-like regularization term weighted by a graph Laplacian. The solution $u_\beta$ in this method reads
\begin{equation}\label{GRaphLA}
    u_{\beta} \in \arg\min_{u \in \mathcal{U}} \left \{\frac{1}{2} \|Au - v^\delta\|^2_2 + \beta \|\Delta_{\Psi_\theta^\delta} u\|_1 \right \},
\end{equation}
where $\Delta_{\Psi_\theta^\delta}$ is a graph Laplacian constructed from an initial reconstruction $\Psi_\theta^\delta$. In this formulation, the graph Laplacian $\Delta$ is fixed after its initial construction based on a preliminary estimate $\Psi_\theta^\delta := \Psi_\theta (v^\delta)$ (where $\Psi$ is a standard reconstruction operator, e.g. FBP, Tikhonov, TV).

\begin{remark}
 The choice of $\Psi$ in this method influences the initial graph structure, but the graph Laplacian itself does not evolve during the solution process.   
\end{remark}

\subsubsection{The \texttt{it-graphLa\(\Psi\)} Method}
\label{app:itgraphla_psi}
The method, proposed in \cite{bianchi2024improved} specifically for acoustic impedance inversion, extends the concept of graph Laplacian regularization by introducing an iterated update of the regularization term. Unlike \texttt{GraphLa+\(\Psi\)}, this method allows the graph Laplacian to evolve alongside the solution. The problem is typically framed as a sequence of minimization problems:

\noindent For each iteration $k$ and regularization parameter $\beta_k$, find $u_{k+1}$ by solving
\begin{equation}\label{itGraphLa}
\begin{cases}
    u_{k+1}&= \arg\min_{u \in \mathcal{U}} \left \{\frac{1}{2} \|Au - v^\delta\|^2_2 + \beta_k \|\Delta_{u_k} u\|_1 \right \} \quad \text{ for \(k \geq 1,\)} \\
    u_0^\delta&:= \Psi(v^\delta),  
\end{cases}
\end{equation}
where \(\Psi : \mathcal{V} \to \mathcal{U}\) is the chosen initialization method.

\noindent Although the work in~\cite{bianchi2024improved} presents numerical validation, it does not provide a rigorous theoretical justification. In contrast, our \texttt{IRMGL+$\Psi$} method introduces a fully iterative and computationally efficient framework that eliminates the need for solving complex minimization problems, as required in~\cite{bianchi2024improved, bianchi2025data}. Crucially, our method addresses this theoretical shortcoming by providing rigorous proofs of convergence and stability. Moreover, the incorporation of the discrepancy principle~\eqref{eqn: discrepancy} enhances the practical effectiveness and reliability of the algorithm.

In the following subsection, we show that the standard Tikhonov, Total Variation, and Filtered Back Projection reconstruction methods satisfy \cref{hypothesis stability}.

\subsection{Initial reconstructors}\label{subsec:initial_rec}
For the numerical experiments, we employ four types of initial reconstructors namely, Filtered Back Projection (FBP), Adjoint (Adj), Tikhonov (Tik), and Total Variation (TV).  
Concerning \cref{hypothesis stability}, the continuity (i.e., Lipschitz property) of the initial reconstructor $\Psi$ plays a central role.  
We first provide a precise definition of the initializers utilized in our experiments, followed by a justification of  \cref{hypothesis stability}.
For a detailed discussion of these reconstruction methods, we refer the reader to \cite{bianchi2025data, kak2001principles,  scherzer2009variational}.

 \noindent \textbf{Tikhonov:}
    Consider the standard Tikhonov reconstruction
\[
\Psi_\lambda(v) := (A^\top A + \lambda I)^{-1} A^\top v, \qquad \lambda > 0.
\]
Let the regularization parameter $\lambda = \lambda(\delta, v^\delta)$ be chosen according to the \emph{discrepancy principle} \cite[Equation~(4.57)]{engl1996regularization} then  
 we have  $\lambda_k \to \lambda$ as $k \to \infty$. Define
\[
T_k := (A^\top A + \lambda_k I)^{-1}, \qquad T := (A^\top A + \lambda I)^{-1}.
\]
Then, we have
\[
\|\Psi_{\lambda_k}(v^{\delta_k}) - \Psi_\lambda(v^\delta)\| \le 
\underbrace{\|(T_k - T) A^\top v^{\delta_k}\|}_{I} +
\underbrace{\|T A^\top (v^\delta - v^{\delta_k})\|}_{II}.
\]
The two terms can be estimated as follows. For the first term
$I$,
    \[
    I \le |\lambda - \lambda_k| \, \|T_k\| \, \|T A^\top\| \, \|v^{\delta_k}\| \le 
    |\lambda - \lambda_k| \, \frac{\|v^{\delta_k}\|}{2 \lambda_k \sqrt{\lambda}} \to 0 \quad \text{as } k \to \infty.
    \] 
 Similarly, using the boundedness of $T A^\top$,
    \[
    II \le \|T A^\top\| \, \|v^\delta - v^{\delta_k}\| \le \frac{1}{2 \sqrt{\lambda}} \, \|v^\delta - v^{\delta_k}\| \to 0 \quad \text{as } k \to \infty.
    \]
Hence, the standard Tikhonov reconstructor satisfies \cref{hypothesis stability} under discrepancy principle.


\noindent \textbf{TV:}
    Consider the proximal TV denoiser
    \[
    \Phi_\lambda(v) :=  \mathrm{prox}_{\lambda \mathrm{TV}}(v) = \arg\min_{u \in \mathcal U} \left \{ \frac{1}{2}\|u-v\|^2 + \lambda \, \mathrm{TV}(u)\right \},
    \]
    where $\mathrm{TV}$ is the total variation functional and $\lambda>0$ is a fixed regularization parameter. By standard convex analysis (see \cite{Bauschke2011}), the proximal mapping of any proper, convex, lower-semicontinuous functional is  $1$-Lipschitz. i.e.,
    \[
    \|\Phi_\lambda(v_1)-\Phi_\lambda(v_2)\| \le \|v_1-v_2\| \quad \forall \; v_1, v_2 \in \mathcal{V}.
    \]
 Now, let $B$ be a linear FBP operator with finite operator norm $\|B\|$ and consider the composite map
    \[
    \Psi_{\mathrm{TV}}(v) := \Phi_\lambda(Bv), \qquad v \in \mathcal V.
    \]
    Using the Lipschitz property of $\Phi_\lambda$, we obtain
    \[
    \|\Psi_{\mathrm{TV}}(v_1)-\Psi_{\mathrm{TV}}(v_2)\| = \|\Phi_\lambda(Bv_1)-\Phi_\lambda(Bv_2)\| \le \|B(v_1-v_2)\| \le \|B\|\,\|v_1-v_2\|.
    \]
   Hence $\Psi_{\mathrm{TV}}$ is Lipschitz with constant $\|B\|$.
 Consequently, for any sequence of noisy data $v^\delta \to v$, we can shows that Assumption~\ref{hypothesis stability} is satisfied.

The remaining reconstruction operators, namely FBP and Adjoint, are inherently linear by \cref{hypo: bounded} and therefore satisfy \cref{hypothesis stability}. Consequently, each selected $\Psi$ in our experiments defines a Lipschitz continuous mapping in this setting.

The following subsection presents a numerical validation of the proposed scheme on a 2D computed tomography (CT) problem.

\subsection{X-ray Computed Tomography (CT)}
It is a widely used medical imaging technique that plays a crucial role in diagnosing various conditions such as internal injuries, tumors, hemorrhages, and bone fractures. A typical CT system consists of an X-ray source that rotates around the subject, emitting beams from a fixed number of angles along a circular or arc-shaped trajectory. As these X-ray beams pass through the internal structures of the body, different tissues absorb varying amounts of radiation depending on their density. The attenuated rays are then captured by a detector. See \cite{natterer2001mathematics} for the nice details.

\noindent The set of all such measurements, acquired at multiple projection angles, is referred to as a \emph{sinogram}. In the discrete setting, this process can be modeled by a linear system
\[
A \in \mathbb{R}^{m \times n}, \quad u \in \mathcal{U} := \mathbb{R}^n, \quad v^\delta \in \mathcal{V} := \mathbb{R}^m,
\]
where \( A \) denotes the discretized Radon transform (or forward operator), \( u \) represents the unknown 2D image vectorized into \( n \) pixels (or voxels), and \( v^\delta \) is the vectorized sinogram obtained from noisy measurement data.

\subsubsection{Experimental setup}\label{experimental setup}
The vectorization of the 2D image is typically performed using a lexicographic (1D) ordering of the pixels \( a = (i,j) \), following either column-major or row-major convention. For an image of size \( E \times F \), the total number of pixels is \( n = E \cdot F \), and the image is reshaped into a vector in \( \mathbb{R}^n \).

\noindent In \cite{jonasadler20181442734}, the Operator Discretization Library (ODL), a Python-based framework, is introduced for modeling X-ray CT measurements and reconstruction. In our experiments, we utilize the function \texttt{odl.uniform\_discr} to define the image domain as a uniformly discretized square grid of size \( E \times F \) over the region \([-128, 128]^2 \subset \mathbb{R}^2\), represented with single-precision floating point values.

\noindent For the numerical simulation, we employ the Shepp–Logan phantom image provided by \texttt{skimage.data}. The acquisition setup is based on a parallel-beam geometry with \( m_\theta = 60 \) uniformly distributed projection angles over the interval \( [0, 2\pi) \), constructed using the \texttt{odl.tomo.parallel\_beam\_geometry} routine. The forward operator \( A \in \mathbb{R}^{m \times n} \), which approximates the discrete Radon transform, is implemented via \texttt{odl.tomo.RayTransform}. Here, \( m = m_\theta \cdot d \), where \( d \) denotes the number of detector elements, which is automatically determined by ODL to ensure full coverage of the image domain. Specifically, for \( E = F = 128 \), we have \( d = \lceil \sqrt{2} \cdot 128 \rceil = 363 \), and hence \( m = 60 \cdot 363 = 21780 \), while the number of image pixels is \( n = 128 \cdot 128 = 16384 \) leading to an ill-conditioned linear system. For the parameter selection in the edge-weight function~\eqref{eqn: edge weight fun}, we set \(R = 6\) and \(\sigma = 0.05\) based on empirical considerations. The step size $\alpha_k^\delta$ and $\beta_k^\delta$ is chosen by \cref{alpha} and \cref{beta} with $\eta_0 =0.2, \eta_1 = 0.5, \nu_0 = \nu_1 = 0.05$ and $\tau = 2$. For these choices the constant $C$ as in \cref{assum: On C} remains positive.
\noindent In order to replicate real-world circumstances, we introduce Gaussian noise $\xi$ into the sinogram at an intensity level of $\delta$. In particular, we compute $v^\delta$ as
\[v^\delta = v +\delta_{rel} \| v \| \frac{\xi}{\| \xi\|},\]
where $\delta_{rel} = \|v - v^\delta\| / \|v\|$ is the relative noise level and $\delta = \delta_{rel}\|v\|$ is the noise level.
Similar to the investigation in \cite{bianchi2025data}, which explores the impact of different reconstruction operators $\Psi$ within variational regularization frameworks, our iterative method also incorporates multiple initialization strategies (excluding the NETT-based approaches \cite{evangelista2023rising, ronneberger2015u}). Specifically, we employ FBP, Tik, and TV initializers. The FBP initialization is implemented using \texttt{odl.tomo.fbp.op} with a Hann filter. For the Tikhonov approach, the regularization parameter is fixed at $50$, and for the TV-based initialization, we utilize the \texttt{denoise.tv.chambolle} function from the \texttt{skimage.restoration} module.
Additionally, we introduce $A^*$ (the adjoint operator) as an initializer. This choice is particularly relevant as $A^*$ inherently satisfies \Cref{hypothesis stability}, providing a fundamental, yet non-regularizing, starting point for the reconstruction.

\subsubsection{Reconstruction results}\label{subsec: reconstrution}
The primary objective of this subsection is to evaluate the reconstruction accuracy under varying noise levels using the discrepancy principle~\eqref{eqn: discrepancy} as the stopping criterion. 
 In Fig.~\ref{fig:true shape logan}, we present the true image of Shepp-Logan Phantom and noisy sinogram.
\begin{figure}[htbp] 
    \centering
    \includegraphics[width=0.8\textwidth]{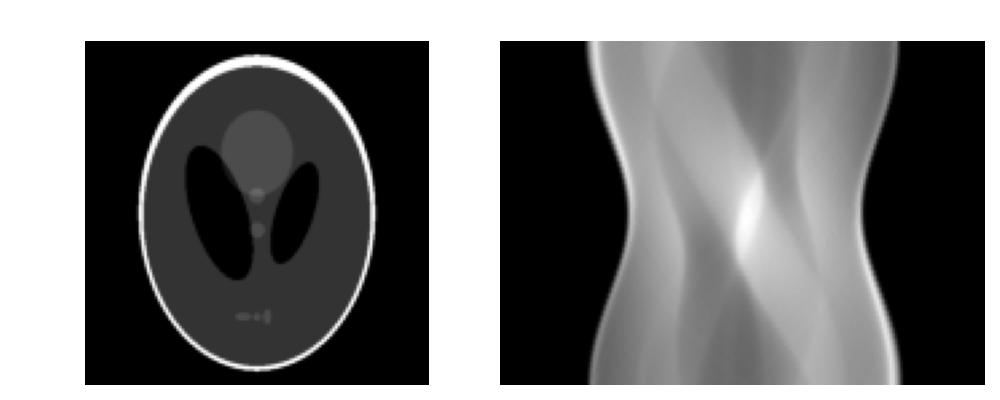} 
    \caption{Left: True phantom, Right: Noisy sinogram \(v^\delta\) with \(\delta = 0.05\). }
    \label{fig:true shape logan}
\end{figure}
\noindent  The maximum number  of iterations is set to \( 1000.\)
The quantitative performance of our reconstructed images is evaluated using relative error, \textit{Peak Signal-to-Noise Ratio (PSNR)} and the \textit{Structural Similarity Index (SSIM)}~\cite{wang2004image}. 
Since  pixel intensities are in $[0, 1]$, the PSNR is given by
\[
\text{PSNR} := 20 \log_{10} \left( \frac{1}{\| u^\dagger - u_{k_\delta}^\delta \|} \right),
\]
and the relative error is defined by
\[
\textbf{RE} := \frac{\|u_{k_\delta}^\delta - u^\dagger\|}{\|u^\dagger\|},
\] 
where $u^\dagger$ denotes the ground truth image and $u_{k_\delta}^\delta$ is the reconstructed image.

\begin{figure} 
    \centering
    \begin{subfigure}[htbp]{\textwidth} 
        \centering
        \includegraphics[width=\textwidth]{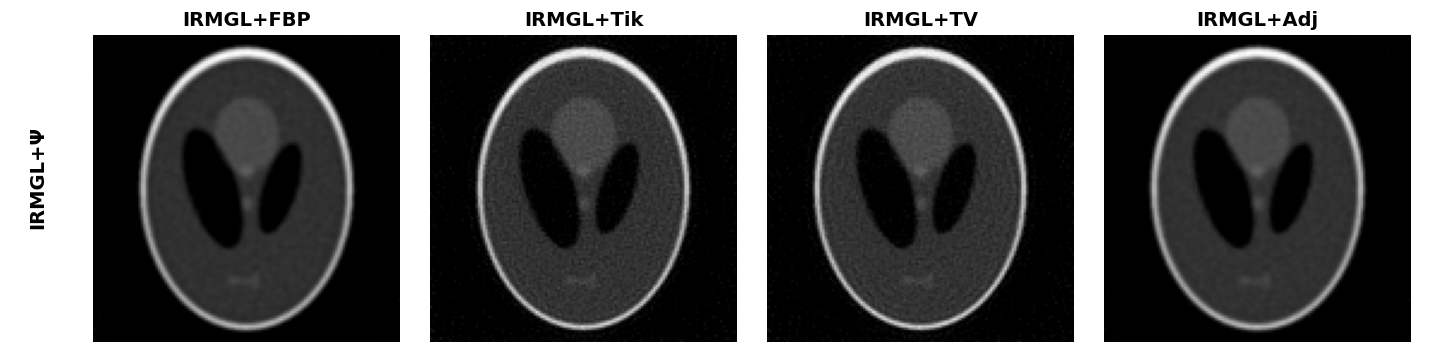}
         \caption{Noise Level: 0.2}
        \end{subfigure}
         \vspace{-0.3cm}
    \begin{subfigure}[htbp]{\textwidth} 
        \centering
        \includegraphics[width=\textwidth]{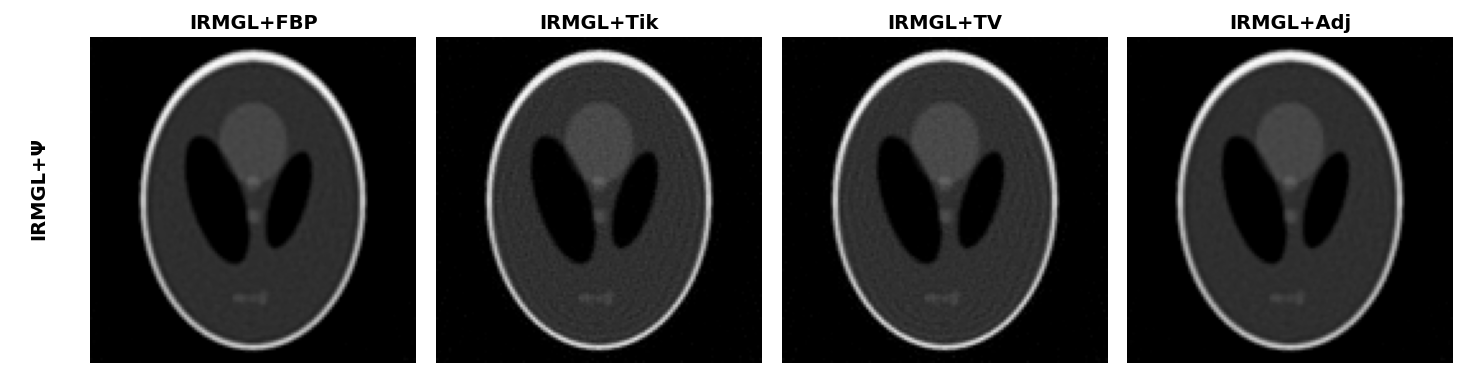}
        \label{fig:rec_0001}
        \caption{Noise Level: 0.1}
    \end{subfigure}
     \vspace{-0.3cm}
    \begin{subfigure}[htbp]{\textwidth}
        \centering
\includegraphics[width=\textwidth,height=0.8\textheight,keepaspectratio]{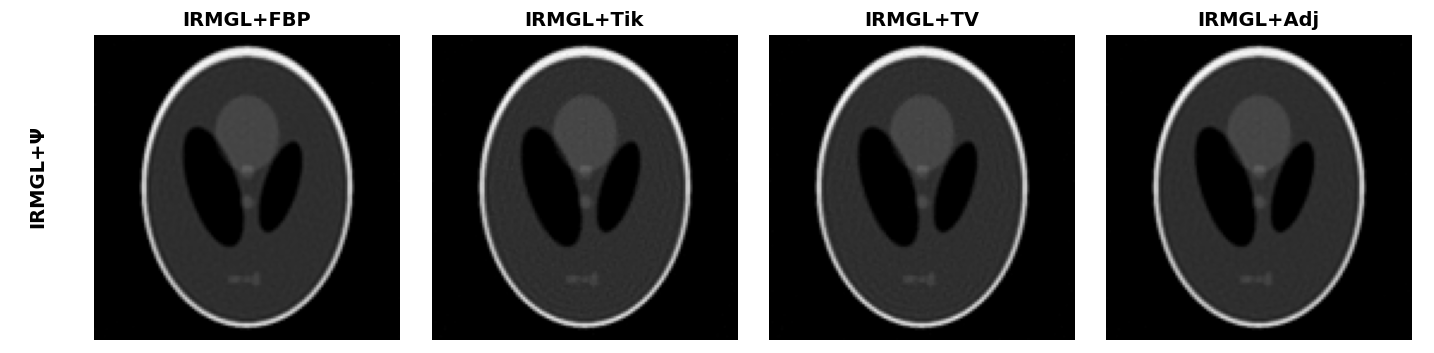}
         \caption{Noise Level: 0.05}
    \end{subfigure}
     \vspace{-0.3cm}
    \begin{subfigure}[htbp]{\textwidth}
        \centering
\includegraphics[width=\textwidth,height=0.8\textheight,keepaspectratio]{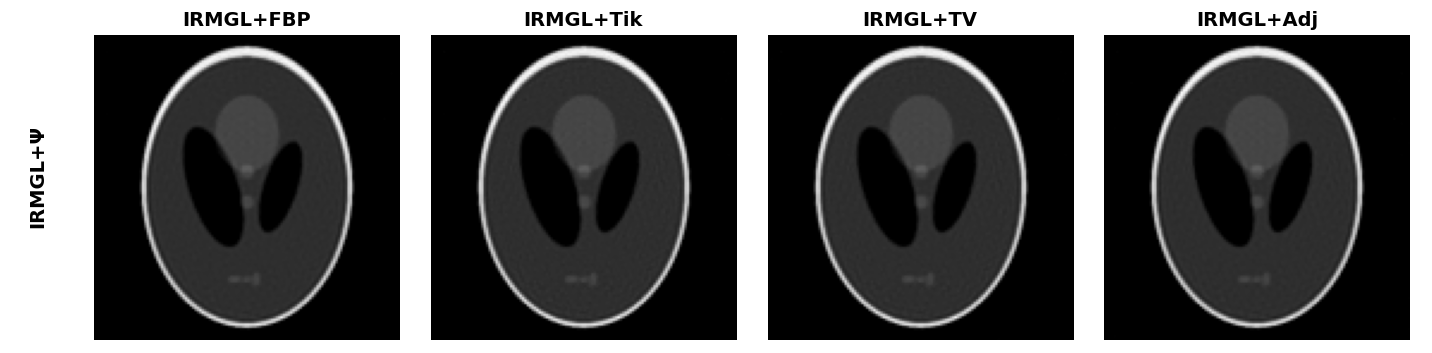}
      \caption{ Noise Level: 0.03}
        \end{subfigure}
        \vspace{-0.3cm}
    \begin{subfigure}[htbp]{\textwidth}
        \centering
\includegraphics[width=\textwidth,height=0.8\textheight,keepaspectratio]{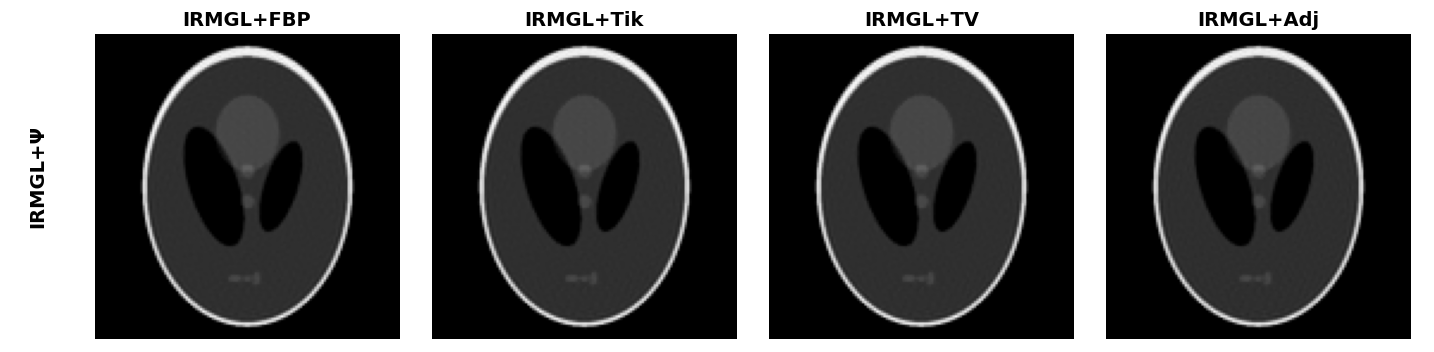}
        \caption{  Noise Level: 0.01}
    \end{subfigure}
    \caption{  Final reconstructions using \texttt{IRMGL+\(\Psi\)} with discrepancy stopping for different noise levels, initialized with various initial reconstructors $\Psi$.}
    \label{fig:combined_reconstructions}
\end{figure}

\begin{figure}[htbp]
    \centering
    \begin{subfigure}[b]{0.48\textwidth}
        \centering
        \includegraphics[width=\textwidth]{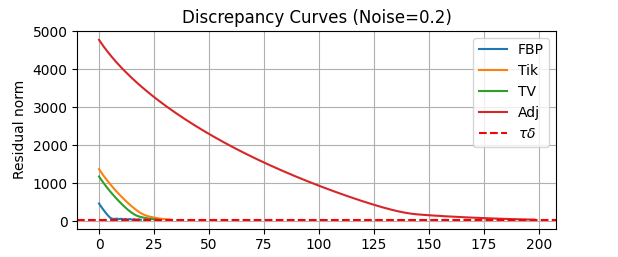}
        \label{fig:dis_0001}
    \end{subfigure}
    \hfill
    \begin{subfigure}[b]{0.48\textwidth}
        \centering
        \includegraphics[width=\textwidth]{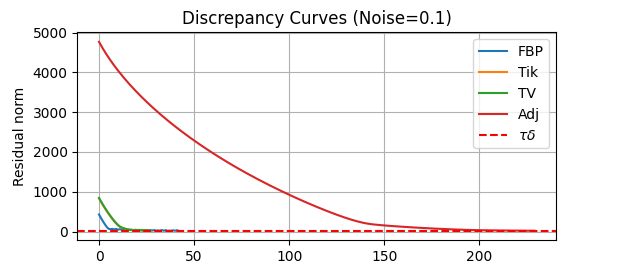}
        \label{fig:dis_0005}
    \end{subfigure}

    \vspace{3mm} 

    \begin{subfigure}[b]{0.48\textwidth}
        \centering
        \includegraphics[width=\textwidth]{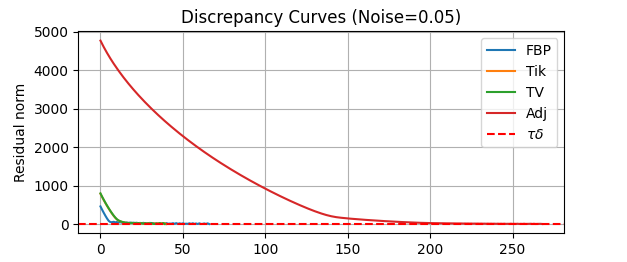}
        \label{fig:dis_001}
    \end{subfigure}
    \hfill
    \begin{subfigure}[b]{0.48\textwidth}
        \centering
        \includegraphics[width=\textwidth]{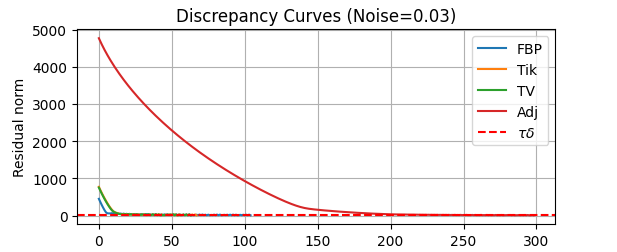}
        \label{fig:dis_005}
    \end{subfigure}

    \vspace{3mm} 

    \begin{subfigure}[b]{0.6\textwidth}
        \centering
        \includegraphics[width=\textwidth]{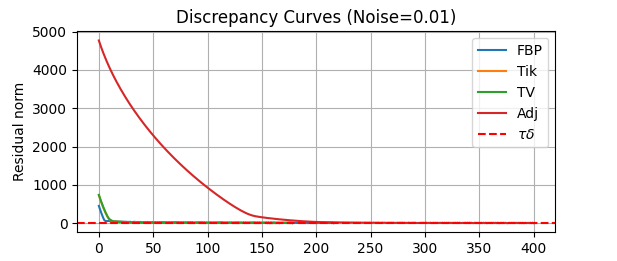}
        \label{fig:dis_01}
    \end{subfigure}
    \caption{Discrepancy curves for different noise levels $\delta_{rel}$.}
    \label{fig:combined_discrepancy}
\end{figure}

\begin{table}
    \centering
    \caption{Numerical performance of \texttt{IRMGL+$\Psi$} for different noise levels $\delta_{rel}$.}
    \label{tab:Discrepancy_table}
    \renewcommand{\arraystretch}{1.2}
    \setlength{\tabcolsep}{6pt}
    \begin{tabular}{c|lccccc}
        \toprule
        $\boldsymbol{\delta_{rel}}$ & \textbf{$\Psi$} & \textbf{Iterations} & \textbf{RE} & \textbf{Discrepancy} & \textbf{PSNR} & \textbf{SSIM} \\
        \midrule

        \multirow{5}{*}{0.2}
           & \texttt{FBP}  & 19 & 0.2142 & 39.99 & 26.03 & 0.9324 \\
            & \texttt{Tik}  & 33 & 0.1729 & 41.07 & 27.90 & 0.7871 \\
            & \texttt{TV}   & 27 & 0.1684 & 41.68  & 28.12 & 0.7893 \\
            & \texttt{Adj} & 198 & 0.2170 & 41.70 & 25.92  & 0.9298 \\
        \midrule
        \multirow{5}{*}{0.1}
            & \texttt{FBP} & 42 & 0.1385 & 19.22 & 29.82 & 0.9634 \\
            & \texttt{Tik}  & 27 & 0.1203 & 18.92 & 31.04 & 0.9217 \\
            & \texttt{TV}   & 27 & 0.1208 & 19.25  & 31.01 & 0.9118 \\
            
            &\texttt{Adj} & 229 & 0.1445 & 20.40 & 29.45  & 0.9618 \\
        \midrule

        \multirow{5}{*}{0.05}
          & \texttt{FBP}         & 66 & 0.0967 & 9.99 & 32.95 & 0.9782 \\
            & \texttt{Tik}   & 40 & 0.0927 & 10.24 & 33.31 & 0.9689 \\
            & \texttt{TV}  & 39 & 0.0918 & 9.98  & 33.39 & 0.9718 \\
            & \texttt{Adj} & 268 & 0.0973 & 10.29 & 32.89  & 0.9777 \\
        \midrule
         \multirow{5}{*}{0.03}
          & \texttt{FBP}       & 104 & 0.0742 & 5.89 & 35.24 & 0.9852 \\
            & \texttt{Tik}   & 67 & 0.0703 & 5.84 & 35.71 & 0.9822 \\
            & \texttt{TV}   & 63 & 0.0722 & 6.03  & 35.48 & 0.9828 \\
            & \texttt{Adj} & 298 & 0.0769 & 6.21 & 34.93  & 0.9837 \\
        \midrule
        \multirow{5}{*}{0.01}
           & \texttt{FBP}        & 216 & 0.0479 & 2.01 & 39.10 & 0.9920 \\
            & \texttt{Tik}   & 166 & 0.0460 & 2.07 & 39.38 & 0.9921 \\
            & \texttt{TV}  & 178 & 0.0455 & 1.87  & 39.48 & 0.9926 \\
            & \texttt{Adj} & 400 & 0.0495 & 2.02 & 38.75  & 0.9899 \\
        \bottomrule
    \end{tabular}
\end{table}
In the context of this model problem, Algorithm~\ref{alg:buildtree} is executed using noisy datasets at various noise levels. The resulting computational performance, including the required number of iterations and the achieved residual values, is summarized in \Cref{tab:Discrepancy_table}. These findings suggest that all four initial reconstructors (FBP, Adj, Tik and TV), when integrated with the \texttt{IRMGL+\(\Psi\)} method, are capable of producing approximate solutions of comparable quality. To visually demonstrate the behavior of the discrepancy principle~\eqref{eqn: discrepancy} as a termination rule for the proposed method, Fig.~\ref{fig:combined_discrepancy} displays the evolution of the residual \(\|Au_{k_\delta}^\delta - v^\delta\|\) with respect to the iteration count \(k\). This analysis is conducted for data with relative noise levels of \(\delta_{rel} = 0.2, 0.1, 0.05, 0.03\), and \(0.01\), with iterations concluding upon satisfaction of the discrepancy principle or reaching a maximum of \(1000\) iterations.

\noindent As depicted in Fig.~\ref{fig:combined_reconstructions}, we present the reconstructed images obtained at various noise levels for each of the four initial regularizers $\Psi$. It is worth noting that although the proposed \texttt{IRMGL+$\Psi$} method exhibits comparable performance and rapid convergence across all choices of initial reconstructor $\Psi$, the initialization with $\Psi = A^*$ generally yields superior reconstruction quality, albeit requiring a larger number of iterations. This behavior can be attributed to the inherent nature of the other initializers (FBP, Tik, and TV) which already incorporate regularization. Consequently, when these reconstructors are used as starting points, the proposed iterative scheme achieves satisfactory results within fewer iterations.

\subsubsection{Comparison results}
For the model problem described above, we implement the \texttt{GraphLa+$\Psi$} method \cref{GRaphLA}, the \texttt{it-GraphLa+$\Psi$} method \cref{itGraphLa}, and our proposed approach under the same experimental setup, with a noise level of $0.02$.  To address the $\ell^2$–$\ell^1$ minimization subproblem that arises in both \cref{GRaphLA} and \cref{itGraphLa}, we employ the majorization–minimization framework in conjunction with a generalized Krylov subspace approach, as proposed in \cite{Lanza2015} and suggested in \cite{bianchi2024improved, bianchi2025data}. For the  method~\cref{itGraphLa}, we perform $k = 20$ outer iterations. The corresponding computational results are presented in \cref{tab:coule_results} and the reconstructed solutions for each methods is shown in Fig.~\ref{fig:comparison}.
\begin{table}
\centering
\caption{Quality comparison on Shepp–Logan phantom.}
\label{tab:coule_results}
\begin{tabular}{lcccc}
\hline
\textbf{Method} &\textbf{Iterations}  & \textbf{RE} & \textbf{SSIM} & \textbf{PSNR} \\
\hline
\texttt{FBP} & & 0.0338 & 0.8988 & 29.43 \\
\texttt{Tik} & & 0.0932 & 0.5233 & 20.50 \\
\texttt{TV}  &  &0.0292 & 0.2995 & 10.68\\
\hline
\multicolumn{4}{l}{\texttt{GraphLa+$\Psi$}} \\
\hline
\texttt{GraphLa+FBP} & & 0.0058 & 0.9880 & 44.75 \\
\texttt{GraphLa+Tik} & & 0.0101 & 0.9570 & 36.35 \\
\texttt{GraphLa+TV}  & & 0.0092 & 0.9534 & 38.70 \\
\texttt{GraphLa+gt} &  & 0.0027  & 0.9974 & 51.52 \\
\hline
\multicolumn{4}{l}{\texttt{it-GraphLa+$\Psi$}} \\
\hline
\texttt{it-GraphLa+FBP} & & 0.0059 & 0.9976 & 44.57 \\
\texttt{it-GraphLa+Tik} & & 0.0180 & 0.9836 & 39.96 \\
\texttt{it-GraphLa+TV}  & & 0.0177 & 0.9843 & 35.03 \\
\hline
\multicolumn{4}{l}{\texttt{IRMGL+$\Psi$}} \\
\hline
\texttt{IRMGL+FBP} & 138 & 0.0602 & 0.9884 & 37.04 \\
\texttt{IRMGL+Tik} & 84 & 0.0607 & 0.9871 & 37.26 \\
\texttt{IRMGL+TV}  & 91 & 0.0630 & 0.9866 & 36.65 \\
\texttt{IRMGL+Adj} & 334 & 0.0634 & 0.9868 & 36.61 \\
\hline

\end{tabular}
\end{table}
\begin{figure}
    \centering
    \includegraphics[width=0.98\textwidth]{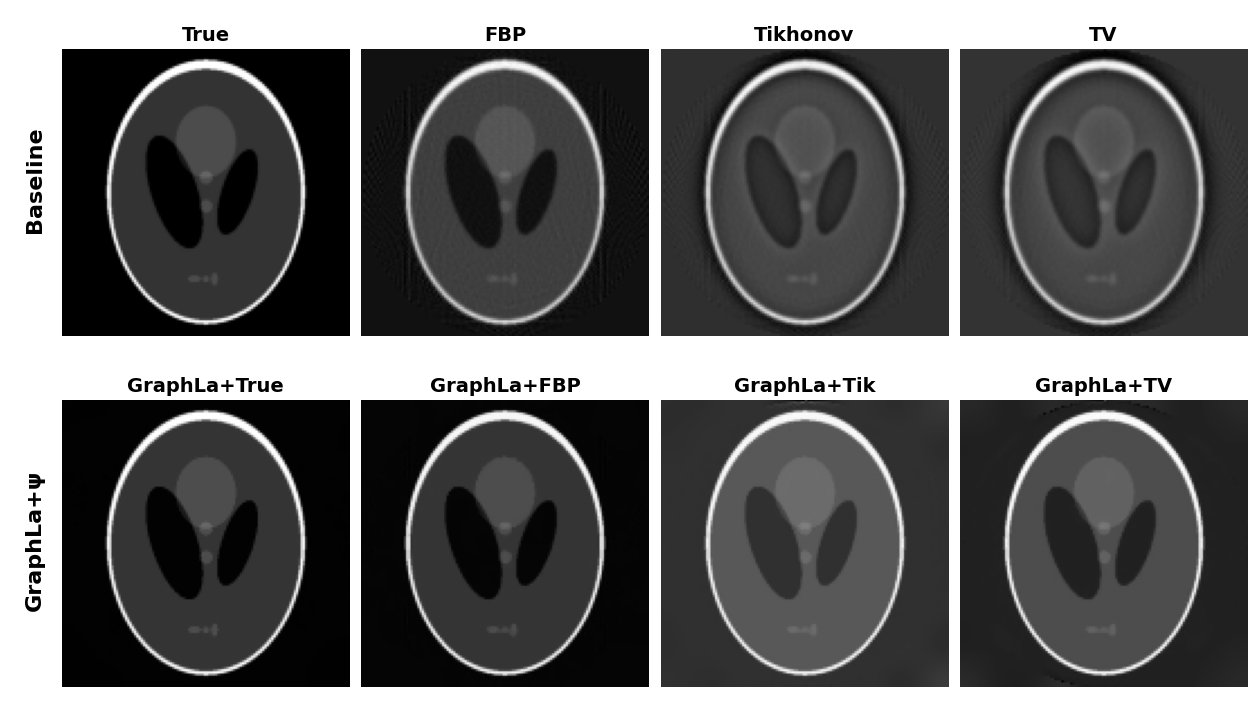}
 \vspace{-0.2cm}
 \includegraphics[width=0.98\textwidth]{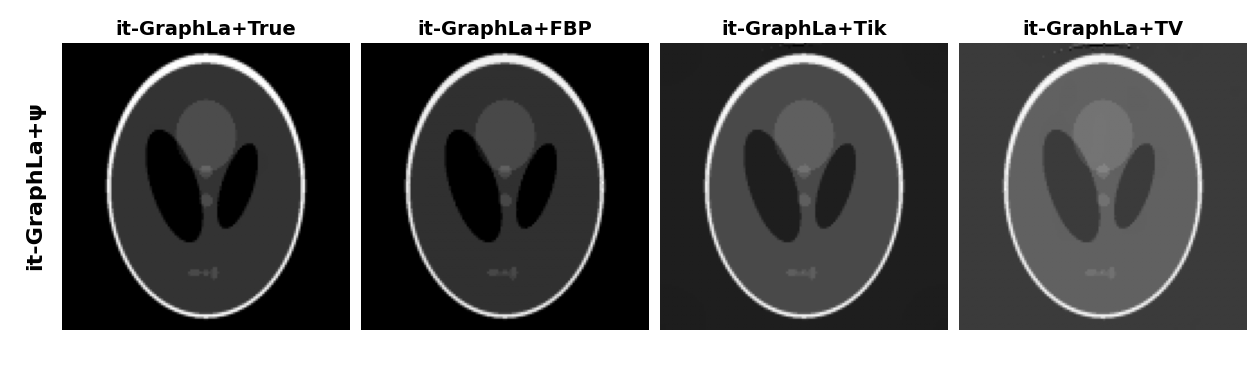}
    \includegraphics[width=0.98\textwidth]{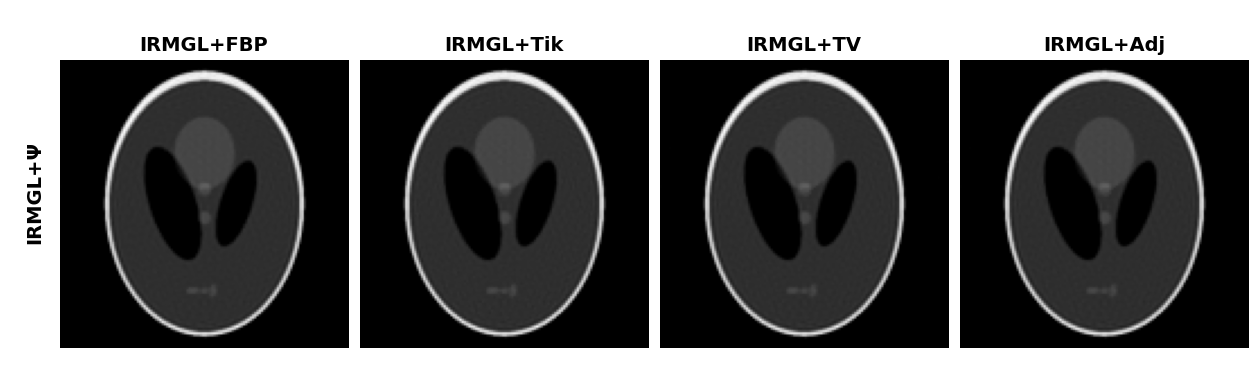}
    \caption{Visual comparison of reconstructed images.}
    \label{fig:comparison}
\end{figure}

In the following subsection, we examine an additional example to further evaluate the performance of the proposed scheme.



\subsection{Image deblurring}

Consider the image deblurring problem \cite{vogel2002computational}, where the forward operator 
$$\mathcal{A}: L^2(\Omega) \to L^2(\Omega), \quad \Omega \subset \mathbb{R}^2,$$
is defined as a spatial convolution with a Gaussian point spread function. Specially, for $u \in L^2 (\Omega)$, the action $\mathcal{A}$ is given by
\[
(\mathcal{A}u)(x, y) = \int_{\Omega} G_\rho(x - x', y - y') \cdot u(x', y') \, dx' \, dy',
\]
where the Gaussian kernel $G_\rho$ is defined as
\[
G_\rho(x, y) = \frac{1}{2\pi \rho^2} \exp\left(-\frac{x^2 + y^2}{2\rho^2}\right),
\]
with $\rho > 0$  controlling the spread (or blurring strength) of the kernel. 

\noindent In the experimental setup, the domain $\Omega$ is discretized into a \( 256 \times 256 \) two-dimensional grid of pixels.  The continuous forward operator \( \mathcal{A} \) is approximated by a discrete convolution operator, denoted by
\[
A : \mathbb{R}^{256 \times 256} \to \mathbb{R}^{256 \times 256}.
\]
Specifically, the action of \( A \) on a discrete image \( u \in \mathbb{R}^{256 \times 256} \) is given by
\[
(Au)_{i,j} = \sum_{k,\ell} G_\rho(i - k, j - \ell) \cdot u_{k,\ell}, \quad i, j, k, \ell \in \{0, \ldots, 255\},
\]
 In practice, this convolution is implemented using \texttt{scipy.ndimage.gaussian\_filter} function in Python.
Let \( u^\dagger \in \mathbb{R}^{256 \times 256} \) denote the true (unknown) image. The observed data \( v^\delta \in \mathbb{R}^{256 \times 256} \) is then modeled as
\[
v^\delta = A u^\dagger + \eta, \qquad \|\eta\| \leq \delta,
\]
where \( \eta \in \mathbb{R}^{256 \times 256} \) represents additive measurement noise, and \( \delta > 0 \) is the known noise level.

\noindent To reconstruct the true image \( u^\dagger \) from the noisy measurement \( v^\delta \), we apply \Cref{alg:buildtree}. In this experiment, we set \( \rho = 1.5 \), determining the severity of the blurring in the forward model. All other experimental parameters are consistent with those specified in~\Cref{experimental setup}. The effectiveness of the proposed method~\eqref{main iterative schrme} under varying noise levels is illustrated in Fig. \ref{fig:ID 0.005}, Fig. \ref{fig:ID 0.001}, and Fig. \ref{fig:ID 0.0005}. Additional quantitative results are provided in~\Cref{tab:experiment table for ID}, which summarizes the reconstruction performance across different noise levels, using the initial guess \( u_0^\delta = A^*(v^\delta) \). The table reports the stopping index \( k_\delta \), the relative error 
along with the corresponding residual norms, PSNR, and SSIM, offering a comprehensive evaluation of the reconstruction quality.
\begin{figure}[htbp]
    \centering
    \includegraphics[width=0.8\textwidth]{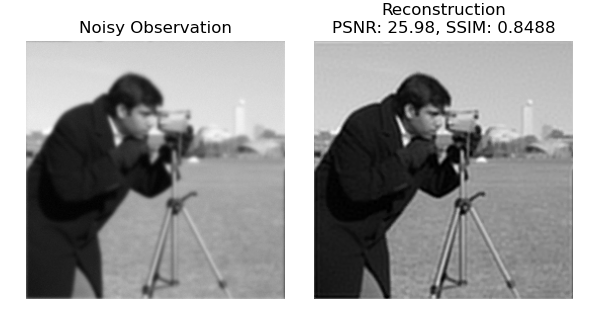}
    \caption{Reconstructed image with $\delta = 0.005$.}
    \label{fig:ID 0.005}
\end{figure}
\begin{figure}[htbp]
    \centering
    \includegraphics[width=0.8\textwidth]{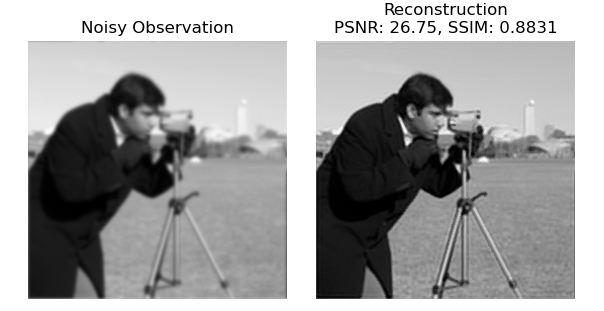}
    \caption{Reconstructed image with $\delta = 0.001$.}
    \label{fig:ID 0.001}
\end{figure}
\begin{figure}[htbp]
    \centering
    \includegraphics[width=0.8\textwidth]{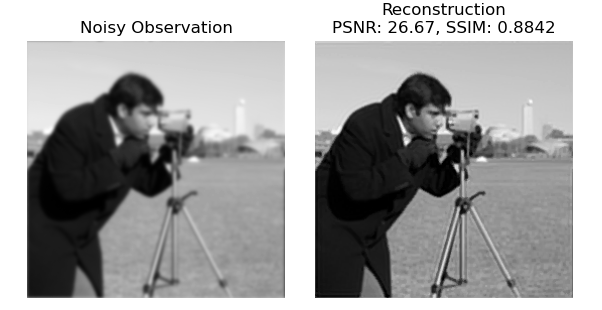}
    \caption{Reconstructed image with $\delta = 0.0005$.}
    \label{fig:ID 0.0005}
\end{figure}

\begin{table}[htbp]
\centering
\caption{Numerical results of~\Cref{alg:buildtree} for image deblurring with initialization \( u_0^\delta = A^*(v^\delta) \).
}
\begin{tabular}{llcccc}
\toprule
$\delta_{rel}$ & \textbf{Iterations} & \textbf{Residual} & \textbf{RE} & \textbf{PSNR} & \textbf{SSIM} \\
\midrule
0.005 & 36         &  1.4060    & 0.0802              & 25.98 & 0.8488 \\
0.003 & 106     & 0.8450        & 0.0821         & 26.46  & 0.8689  \\
0.001 & 200        & 0.4167  & 0.0812          & 26.75 & 0.8831  \\
0.0005 &  287        & 0.3597              & 0.0801  & 26.67  & 0.8842\\
\bottomrule
\end{tabular}
\label{tab:experiment table for ID}
\end{table}
\subsection{Computational Cost and Scalability}
We compare the computational complexity of the  \texttt{IRMGL+$\Psi$} scheme with the variational framework \texttt{graphLa+$\Psi$}. 
Let \(q\) denote the number of image unknowns (pixels) and \(|S|\) the number of edges in the image graph, satisfying \(|S| = \mathcal{O}(qd)\) for an average node degree \(d\). 
Assuming the forward and adjoint evaluations of the operator \(A\) cost \(\mathcal{O}(\mathcal{C}_A)\) and the construction of the graph Laplacian from an image \(u\) costs \(\mathcal{O}(\mathcal{C}_{\mathrm{build}})\), one iteration of \texttt{IRMGL+$\Psi$} requires
\[
\mathcal{C}_{\mathrm{iter}}^{(\mathrm{IRMGL})} 
= 2\mathcal{C}_A + \mathcal{O}(|S|) + \frac{\mathcal{C}_{\mathrm{build}}}{p},
\]
where \(p\) denotes the graph update frequency (\(p=1\) if the graph is rebuilt at every step and \(p=\infty\) if fixed). 
As both \(\mathcal{C}_A\) and \(|S|\) scale linearly with \(q\) for structured operators, the overall complexity per iteration is \(\mathcal{O}(q)\). 
All core operations (matrix vector multiplications with sparse matrices) are highly parallelizable, ensuring good scalability and low memory overhead.

In contrast, the \texttt{graphLa+$\Psi$} approach minimizes a non-smooth convex functional involving an \(\ell_1\)-type penalty. 
Standard optimization solvers such as ADMM or primal--dual hybrid gradient require \(t_{\mathrm{outer}}\) outer and \(r\) inner iterations, leading to a total cost
\[
\mathcal{C}_{\mathrm{solve}}^{(\mathrm{var})}
\simeq t_{\mathrm{outer}}\, r\, (\mathcal{C}_A + \mathcal{O}(|S|)).
\]
Even for moderate iteration counts (\(t_{\mathrm{outer}}\approx 100\), \(r\approx 50\)), this cost exceeds that of the iterative method by two to three orders of magnitude.  

In summary, while the variational model may yield highly accurate reconstructions, it entails substantial computational effort. 
The proposed \texttt{IRMGL+$\Psi$} method, with its linear per-iteration complexity, efficient parallel implementation, and tunable graph update frequency, offers a significantly more scalable and memory-efficient alternative for large-scale and high-resolution imaging problems.

\section{Concluding remarks}
\label{sec:conclusions}
This work explores the use of a graph Laplacian operator, built from an initial approximation of the solution obtained via a reconstructor \( \Psi \), as a regularizer in iterative regularization methods for solving linearly ill-posed problems. In this method, we build a new graph Laplacian at each step, based on the previous approximation. Under certain moderate assumptions on the forward operator $A$ and generic assumptions on \( \Psi \), we showed that the resulting method is a stable and convergent regularization method. As a stopping rule for the approach, we examined the discrepancy principle. Irrespective of whatever reconstructor  \( \Psi \) was used initially, \texttt{IRMGL+\(\Psi\)} greatly enhanced the quality of the reconstructions in all of the numerical experiments that were presented.

  \noindent  There are numerous open research questions. Since this is the initial attempt to incorporate the graph Laplacian into the iterative regularization method, additional approaches such as Gauss-Newton~\cite{qi2000iteratively}, adaptive heavy ball~\cite{jin2024adaptive}, and stochastic gradient descent methods~\cite{jahn2020discrepancy,jin2020convergence, jin2023convergence} may also be investigated. Furthermore, as a preliminary reconstruction, architectures based on deep neural networks (DNNs) can be examined. 
%
\section*{Acknowledgments}
We sincerely thank the anonymous referees for their meticulous review and constructive feedback, which substantially improved the clarity and contribution of this paper.
\section*{Data Availability}
The datasets and source code developed in the course of this research are available from the authors upon reasonable request.

\end{document}